\documentclass[onefignum,onetabnum]{siamart251216}

\usepackage{amsmath,amsfonts,amsopn,amssymb}
\usepackage{graphicx}
\usepackage[caption=false]{subfig}
\usepackage{multirow,relsize,makecell,diagbox}
\usepackage{url,booktabs,cite}
\ifpdf
\DeclareGraphicsExtensions{.eps,.pdf,.png,.jpg}
\else
\DeclareGraphicsExtensions{.eps}
\fi
\usepackage{algorithm, algorithmic}
\usepackage[shortlabels]{enumitem}
\usepackage{accents,stmaryrd}

\numberwithin{theorem}{section}
\newsiamremark{assumption}{Assumption}
\newsiamremark{remark}{Remark}
\newsiamremark{question}{Question}
\newsiamremark{example}{Example}
\crefname{assumption}{Assumption}{Assumptions}
\crefname{remark}{Remark}{Remarks}
\crefname{example}{Example}{Examples}

\headers{Preconditioners for mixed stress Stokes}{Kaibo Hu, Jongho Park, and Jindong Wang}

\title{Optimal block preconditioners\\ for a mass-conserving mixed stress formulation\\ of Stokes flow\thanks{Submitted to the editors DATE.
\funding{The work of KH was supported by a Royal Society University Research Fellowship (URF$\backslash$R1$\backslash$221398) and an ERC Starting Grant (project 101164551, GeoFEM). The work of JW was supported by a Royal Society Newton International Fellowship (NIF$\backslash$R1$\backslash$252881). Views and opinions expressed are however those of the authors only and do not necessarily reflect those of the European Union or the European Research Council. Neither the European Union nor the granting authority can be held responsible for them.}
}}

\author{
Kaibo Hu\thanks{Mathematical Institute, University of Oxford, Oxford, OX2 6GG, United Kingdom (\email{kaibo.hu@maths.ox.ac.uk}, \email{jindong.wang@maths.ox.ac.uk}).}
\and
Jongho Park\thanks{School of Science and Engineering, The Chinese University of Hong Kong, Shenzhen, Guangdong 518172, China
 (\email{jonghopark@cuhk.edu.cn}).}
\and
Jindong Wang\footnotemark[2]
}

\ifpdf
\hypersetup{
  pdftitle={Optimal block preconditioners for a mass-conserving mixed stress formulation of Stokes flow},
  pdfauthor={Kaibo Hu, Jongho Park, and Jindong Wang}
}
\fi

\usepackage[normalem]{ulem}
\usepackage{color}

\begin{document}

\maketitle
\begin{abstract}
We present optimal block diagonal and triangular preconditioners for a mass-conserving mixed stress formulation of Stokes flow.
The algebraic formulation leads to a double saddle point system with unknowns corresponding to discrete stress, velocity, vorticity, and pressure.
MINRES equipped with a block diagonal preconditioner for an augmented Lagrangian formulation of this system is analyzed and shown to be optimal, in the sense that the convergence rate is independent of key parameters such as mesh size and kinematic viscosity.
GMRES equipped with a block triangular preconditioner is also analyzed using a field-of-values approach.
Finally, we present numerical results for both two- and three-dimensional model problems to validate the parameter robustness of the proposed preconditioners.
\end{abstract}

% Keywords
\begin{keywords}
mass-conserving mixed stress formulation,
block preconditioners,  
double saddle-point systems,
Stokes equations,
field-of-values analysis
\end{keywords}

% AMS classification
\begin{AMS}
65F08,  % Preconditioners for iterative methods
65N22,  % Numerical solution of discretized equations for boundary value problems involving PDEs
65N30,  % Finite element, Rayleigh-Ritz and Galerkin methods for boundary value problems involving PDEs
76D07  	% Stokes and related (Oseen, etc.) flows
\end{AMS}

% Section: Introduction
\section{Introduction}
\label{Sec:Introduction}
In this paper, we develop optimal block preconditioners for a finite element discretization of a mass-conserving mixed stress~(MCS) formulation of the Stokes problem with symmetric viscous stress. Let $\Omega \subset \mathbb{R}^d$~($d=2,3$) be a bounded polyhedral domain. Let $u$ and $p$ denote the velocity and pressure, respectively. Given an external body force $f \colon \Omega \to \mathbb{R}^d$ and a positive kinematic viscosity ${\nu} \colon \Omega \to \mathbb{R}$, the velocity--pressure formulation of the Stokes system reads
\begin{equation}
\label{strong}
\begin{aligned}
- \operatorname{div} ({\nu} \varepsilon (u) ) + \nabla p = f \quad &\text{ in } \Omega, \\
\operatorname{div} u = 0 \quad &\text{ in } \Omega, \\
u = 0 \quad &\text{ on } \partial \Omega,
\end{aligned}
\end{equation}
where
\begin{equation*}
\varepsilon (u) = \frac{\nabla u + (\nabla u)^T }{2}.
\end{equation*}

Over the past decades, a class of intrinsic methods exhibiting pullback invariance has emerged, and these methods therefore extend naturally to curved manifolds. Examples include the Regge element \cite{CH:2011,CL:2026}, arising from Regge calculus and the elasticity complex, the tangential-displacement normal-normal-stress continuous~(TDNNS) method \cite{PS:2011,CH:2007} for elasticity, and mass-conserving mixed stress~(MCS) methods \cite{GJPS:2020a,GJPS:2020b} for the Stokes problem. The Stokes problem has motivated the development of structure-preserving discretization methods. In this work, we adopt the MCS formulation of \cite{GJPS:2020b}, which enjoys several attractive structural properties. In particular, the discrete velocity is exactly divergence-free, which implies exact mass conservation and pressure robustness. Moreover, the formulation delivers optimal approximation properties for both the pressure and the stress variable. To weakly enforce the symmetry of the viscous stress, the method introduces a vorticity variable and employs a relaxed $H(\operatorname{div})$-conforming velocity space~(cf.~\cite{ABMXZ:2014}).

% In~\cite{GJPS:2020b}, a robust finite element discretization of~\eqref{strong} was proposed; see also~\cite{GJPS:2020a}. This discretization is based on the MCS formulation and enjoys several attractive structural properties. In particular, the discrete velocity is exactly divergence-free, which implies exact mass conservation and pressure robustness. Moreover, the formulation delivers optimal approximation properties for both the pressure and the stress variable. To weakly enforce the symmetry of the viscous stress, the method introduces a vorticity variable and employs a relaxed $H(\operatorname{div})$-conforming velocity space~(cf.~\cite{ABMXZ:2014}).

Despite these attractive features, the resulting algebraic system is significantly more complicated than those arising from standard velocity--pressure discretizations. In particular, the discrete MCS Stokes system is a double saddle point problem involving four coupled unknowns: the stress $\sigma$, velocity $u$, pressure $p$, and vorticity $\omega$. Designing efficient iterative solvers for such systems that are robust with respect to both the physical parameters and the mesh size is therefore a challenging task.

In this paper, we propose optimal block diagonal and block triangular preconditioners for the discrete MCS Stokes system, to be used with MINRES~\cite{PS:1975} and GMRES~\cite{SS:1986}, respectively. Our construction is based on an equivalent augmented Lagrangian formulation~\cite{BO:2011,BL:2012,FMW:2019,HVK:2018}, which enables us to exploit the discrete inf--sup condition established in~\cite{GJPS:2020b}. This augmented formulation provides a convenient framework for the design and analysis of robust block preconditioners.

For the block diagonal preconditioners, we first construct an ideal block diagonal preconditioner whose diagonal blocks are given by suitable Schur complements and prove uniform spectral bounds for the preconditioned operator. This construction is in the same spirit as existing algebraic approaches for double saddle point systems; see, for example,~\cite{AB:2018,BBH:2024,HM:2019,XL:2020}. We then prove that each of these Schur complements is spectrally equivalent to suitable mass matrices, leading to a practically computable block diagonal preconditioner. To further improve computational efficiency, we derive inexpensive yet still spectrally equivalent variants by combining mass lumping~\cite{VSB:2023,Wathen:1987} with auxiliary space preconditioning~\cite{PX:2025,Xu:1996}.

For the block triangular preconditioners, we again begin with an ideal preconditioner whose diagonal blocks are Schur complements. We prove, using field-of-values~(FOV) analysis~\cite{LW:2004,EE:2001,Starke:1997}, that this block triangular preconditioner yields an optimal preconditioner for GMRES. We then show that optimality is preserved when the Schur complements are replaced by spectrally equivalent approximations, provided a mild additional condition is satisfied. This leads to practical block triangular preconditioners that retain the robustness predicted by the theory.

We also place our work in the broader context of block preconditioning for fluid and multiphysics systems. Robust block diagonal and triangular preconditioners have been developed for magnetohydrodynamics~\cite{MHHX:2016}, while robust block diagonal preconditioners for poroelasticity were studied in~\cite{CHXY:2020}. Particularly relevant to the present work are the block preconditioners developed for hybridized discontinuous Galerkin discretizations of the Stokes problem in~\cite{RW:2018,RW:2022}.
Compared with these works, our approach has several distinguishing features. First, the underlying algebraic system is a double saddle point problem, which is intrinsically more difficult to solve~\cite{AB:2018,BBH:2024,HM:2019,XL:2020}. Second, the $H(\operatorname{div})$-conforming structure of the velocity discretization~\cite{ABMXZ:2014} necessitates the use of auxiliary space preconditioning techniques~\cite{PX:2025,Xu:1996}. Third, our analysis includes a rigorous FOV theory for the block triangular preconditioners.

The remainder of the paper is organized as follows. In \cref{Sec:MCS}, we review the MCS formulation of the Stokes problem and its finite element discretization. In \cref{Sec:Diagonal}, we develop optimal block diagonal preconditioners for preconditioned MINRES. In \cref{Sec:Triangular}, we develop optimal block triangular preconditioners for preconditioned GMRES. In \cref{Sec:Numerical}, we present numerical results for the proposed preconditioners. Finally, in \cref{Sec:Conclusion}, we conclude the paper with closing remarks.

% Subsection: Notation
\subsection{Notation}
Here, we introduce notation used throughout the paper.
% In~\eqref{strong}, we define
% \begin{equation*}
%     \nu = 2 \tilde{\nu}.
% \end{equation*}
We use the notation $\| \cdot \|_{L^2(T)}$ for all $L^2$-norms on a region $T$, whether scalar-, vector-, or tensor-valued. Additionally, for any matrix $A$, we use $\|A\|$ to denote the $\ell^2$ matrix norm, which corresponds to its spectral norm.

In all $\sup$ and $\inf$ expressions, arguments that make the denominator zero are excluded.

For two scalars $x$ and $y$, we write $x \lesssim y$, or equivalently $y \gtrsim x$, to mean that there exists a constant $C > 0$, independent of the problem parameter $\nu$ and the discretization parameter $h$, such that $x \leq C y$.
Furthermore, we write $x \eqsim y$ if both $x \lesssim y$ and $x \gtrsim y$ hold.

For two symmetric positive definite (SPD) matrices $X$ and $Y$ of the same size, we define the notation analogously: we write $X \lesssim Y$ if there exists a constant $C > 0$, independent of the problem parameter $\nu$ and the discretization parameter $h$, such that
$(Xv, v) \le C (Yv, v)$ for all $v$,
and the relations $\gtrsim$ and $\eqsim$ are defined in the same way.

For an SPD matrix $X$, let $(\cdot,\cdot)_X$ be the inner product induced by $X$ and let $\|\cdot\|_X$ be the associated norm.

% Section: MCS formulation
\section{MCS formulation}
\label{Sec:MCS}
In this section, we present a brief summary of the MCS formulation for the Stokes problem~\eqref{strong} and its finite element discretization introduced in~\cite{GJPS:2020b,GJPS:2020a}.
In what follows, as in~\cite{GJPS:2020b,GJPS:2020a}, for simplicity of exposition we assume that ${\nu}$ is a positive constant on $\Omega$.

Starting from~\eqref{strong}, we introduce the trace-free viscous stress
\begin{equation*}
    \sigma = \nu \varepsilon(u)
\end{equation*}
and the skew-symmetric vorticity variable $\omega$ such that
\[
\nabla u = \varepsilon(u) + \omega.
\]
This yields the following equivalent formulation, known as the MCS formulation~\cite{GJPS:2020b}:
\begin{equation}
\label{MCS_strong}
\begin{aligned}
\frac{1}{\nu}\operatorname{dev}\sigma - \nabla u + \omega &= 0
\quad &&\text{in } \Omega, \\
\operatorname{div}\sigma - \nabla p &= -f
\quad &&\text{in } \Omega, \\
\sigma - \sigma^T &= 0
\quad &&\text{in } \Omega, \\
\operatorname{div}u &= 0
\quad &&\text{in } \Omega, \\
u &= 0
\quad &&\text{on } \partial\Omega,
\end{aligned}
\end{equation}
where the deviatoric part of a matrix $\varsigma \in \mathbb{R}^{d \times d}$ is defined as
\[
\operatorname{dev}\varsigma
=
\varsigma - \frac{1}{d} (\operatorname{tr}\varsigma) I.
\]

% Subsection: Finite element discretization
\subsection{Finite element discretization}
We write
\[
\mathbb{M} := \mathbb{R}^{d\times d},
\quad
\mathbb{K} := \{\gamma \in \mathbb{M} : \gamma^T = -\gamma\}.
\]
Let $\mathcal{T}_h$ be a quasi-uniform simplicial mesh of $\Omega$, and let
$\mathcal{F}_h = \mathcal{F}_h^{\mathrm{int}} \cup \mathcal{F}_h^{\mathrm{bd}}$
denote the sets of interior and boundary facets. For every interior facet $F$, fix a unit normal $n_F$ and label its neighboring elements $T^+$ and $T^-$ consistently with this orientation. For an elementwise field $z$, we set
\[
\llbracket z \rrbracket
:=
 z|_{T^+}-z|_{T^-}
\quad \text{on }F\in\mathcal{F}_h^{\mathrm{int}},
\quad
\llbracket z \rrbracket
:=
 z|_T
\quad \text{on }F\in\mathcal{F}_h^{\mathrm{bd}}.
\]
On an interior facet, all normal and tangential traces below are taken with respect to $n=n_F$; on a boundary facet, $n$ is the outward unit normal. For a vector $v$ and a matrix $\tau$, define
\[
v_n := v \cdot n,
\quad
v_t := v-v_n n,
\quad
\tau_{nn} := n^T\tau n,
\quad
\tau_{nt} := \tau n-\tau_{nn}n.
\]

For any integer $k\geq1$, the discrete velocity, pressure, and vorticity spaces are
\begin{equation}
\label{spaces}
\begin{aligned}
V_h &= H_0(\operatorname{div};\Omega)\cap\mathcal{RT}^k(\mathcal{T}_h), \\
Q_h &= L_0^2(\Omega)\cap\mathbb{P}^k(\mathcal{T}_h), \\
W_h &= \mathbb{P}^k(\mathcal{T}_h,\mathbb{K}),
\end{aligned}
\end{equation}
where
\[
L_0^2(\Omega)
:=
\left\{q\in L^2(\Omega):\int_\Omega q\,dx=0\right\},
\]
and $\mathcal{RT}^k(\mathcal{T}_h)$ is the conforming Raviart--Thomas space of order $k$~\cite{BBF:2013}.

The basic nonconforming stress space is
\[
\Sigma_h
=
\left\{
\tau_h\in\mathbb{P}^k(\mathcal{T}_h,\mathbb{M}):
\operatorname{tr}(\tau_h)=0,\ 
\llbracket(\tau_h)_{nt}\rrbracket=0
\ \text{for all }F\in\mathcal{F}_h^{\mathrm{int}}
\right\}.
\]
The stable stress enrichment introduced in~\cite{GJPS:2020b} is defined as follows. For each simplex $T\in\mathcal{T}_h$ with barycentric coordinates $\{\lambda_i\}$, let $\mathbb{P}_\perp^k(T,\mathbb{K})$ be the $L^2(T,\mathbb{K})$-orthogonal complement of $\mathbb{P}^{k-1}(T,\mathbb{K})$ in $\mathbb{P}^k(T,\mathbb{K})$, and set
\[
\mathbb{P}_\perp^k(\mathcal{T}_h,\mathbb{K})
=
\prod_{T\in\mathcal{T}_h}\mathbb{P}_\perp^k(T,\mathbb{K}).
\]
Define the matrix bubble
\[
B
=
\begin{cases}
\displaystyle
\sum_{i=0}^{3}
\lambda_{i-3}\lambda_{i-2}\lambda_{i-1}\,
\nabla\lambda_i\otimes\nabla\lambda_i,
& d=3, \\
\lambda_0\lambda_1\lambda_2,
& d=2,
\end{cases}
\]
where the indices are understood modulo $4$ when $d=3$, and let
\[
\delta\Sigma_h
=
\left\{
\operatorname{dev}\bigl(\operatorname{curl}(\operatorname{curl}(r_h)B)\bigr):
r_h\in\mathbb{P}_\perp^k(\mathcal{T}_h,\mathbb{K})
\right\}.
\]
Here $\operatorname{curl}$ is taken row-wise on matrix-valued fields, and in $d=2$ the outer $\operatorname{curl}$ is matrix-valued. The total stress space is
\[
\Sigma_h^+ = \Sigma_h\oplus\delta\Sigma_h.
\]

The finite element formulation of~\eqref{MCS_strong} reads: find
$(\sigma_h,u_h,\omega_h,p_h)\in\Sigma_h^+\times V_h\times W_h\times Q_h$ such that
\begin{subequations}
\label{MCS_FEM}
\begin{equation}
\begin{aligned}
a(\sigma_h,\tau)+b_2(\tau,(u_h,\omega_h)) &= 0
&&\text{for all }\tau\in\Sigma_h^+, \\
b_2(\sigma_h,(v,\gamma))+b_1((v,\gamma),p_h) &= -(f,v)
&&\text{for all }(v,\gamma)\in U_h, \\
b_1((u_h,\omega_h),q) &= 0
&&\text{for all }q\in Q_h,
\end{aligned}
\end{equation}
where
\begin{equation}
U_h=V_h\times W_h,
\end{equation}
and
\begin{equation}
\resizebox{\textwidth}{!}{$\displaystyle
\begin{aligned}
a(\varsigma,\tau)
&=
\int_\Omega \nu^{-1}\operatorname{dev}\varsigma:\operatorname{dev}\tau\,dx,
\quad &&\varsigma,\tau\in\Sigma_h^+, \\
b_1((v,\gamma),q)&=b_1((v,0),q)
=
\int_\Omega q\,\operatorname{div}v\,dx,
\quad &&(v,\gamma)\in U_h,\ q\in Q_h, \\
b_2(\tau,(v,\gamma))
&=
\sum_{T\in\mathcal{T}_h}\int_T
(\operatorname{div}\tau\cdot v+\tau:\gamma)\,dx
-
\sum_{F\in\mathcal{F}_h}\int_F
\llbracket\tau_{nn}\rrbracket v_n\,ds \\
&=
-
\sum_{T\in\mathcal{T}_h}\int_T\tau:(\nabla v-\gamma)\,dx
+
\sum_{F\in\mathcal{F}_h}\int_F\tau_{nt}\cdot\llbracket v_t\rrbracket\,ds,
\quad &&\tau\in\Sigma_h^+,\ (v,\gamma)\in U_h.
\end{aligned}
$}
\end{equation}
\end{subequations}
We also write $b_1(v,q)$ for $b_1((v,0),q)$. The well-posedness theorem and its discrete corollary in~\cite[Theorem~4.14 and Corollary~4.17]{GJPS:2020b} imply that~\eqref{MCS_FEM} has a unique solution. Since $\operatorname{div}V_h=Q_h$, its velocity satisfies $\operatorname{div}u_h=0$ pointwise and is therefore exactly mass-conserving.

% Subsection: Useful properties
\subsection{Useful properties}
We equip the discrete spaces with the norms
\begin{equation}
\label{norms}
\begin{aligned}
\|\tau\|_{\Sigma_h^+}^2
&=\|\tau\|_{L^2(\Omega)}^2,
&&\quad \tau\in\Sigma_h^+, \\
\|v\|_{V_h}^2
&=\sum_{T\in\mathcal{T}_h}\|\varepsilon(v)\|_{L^2(T)}^2
+\sum_{F\in\mathcal{F}_h}h^{-1}\|\llbracket v_t\rrbracket\|_{L^2(F)}^2,
&&\quad v\in V_h, \\
\|(v,\gamma)\|_{U_h}^2
&=\|v\|_{V_h}^2
+\sum_{T\in\mathcal{T}_h}\|\kappa(\operatorname{curl}v)-\gamma\|_{L^2(T)}^2,
&&\quad (v,\gamma)\in U_h, \\
\|q\|_{Q_h}^2
&=\|q\|_{L^2(\Omega)}^2,
&&\quad q\in Q_h,
\end{aligned}
\end{equation}
where
\[
\kappa(\phi)
=
\frac12
\begin{bmatrix}
0&-\phi\\
\phi&0
\end{bmatrix}
\quad\text{if }d=2,
\quad
\kappa(\phi)
=
\frac12
\begin{bmatrix}
0&-\phi_3&\phi_2\\
\phi_3&0&-\phi_1\\
-\phi_2&\phi_1&0
\end{bmatrix}
\quad\text{if }d=3.
\]
The local vorticity-minimization property proved in~\cite[Lemma~4.8]{GJPS:2020b} yields
\begin{equation}
\label{norm_V_h_min}
\|v\|_{V_h}
=
\min_{\gamma\in W_h}\|(v,\gamma)\|_{U_h},
\quad v\in V_h.
\end{equation}
The MCS continuity estimates proved in~\cite[Lemma~4.15]{GJPS:2020b} give
\begin{subequations}
\label{continuity}
\begin{align}
\label{continuity_A}
|a(\varsigma,\tau)|
&\leq \nu^{-1}\|\varsigma\|_{\Sigma_h^+}\|\tau\|_{\Sigma_h^+},
&&\quad \varsigma,\tau\in\Sigma_h^+, \\
\label{continuity_B_1}
|b_1((v,0),q)|
&\lesssim \|v\|_{V_h}\|q\|_{Q_h}
\lesssim \|(v,0)\|_{U_h}\|q\|_{Q_h},
&&\quad v\in V_h,\ q\in Q_h, \\
\label{continuity_B_2}
|b_2(\tau,(v,\gamma))|
&\lesssim \|\tau\|_{\Sigma_h^+}\|(v,\gamma)\|_{U_h},
&&\quad \tau\in\Sigma_h^+,\ (v,\gamma)\in U_h.
\end{align}
\end{subequations}
Because every $\tau\in\Sigma_h^+$ is trace-free,
\begin{equation}
\label{coercivity_A}
a(\tau,\tau)
=
\int_\Omega\nu^{-1}\tau:\tau\,dx
=
\nu^{-1}\|\tau\|_{\Sigma_h^+}^2,
\quad \tau\in\Sigma_h^+.
\end{equation}
The discrete stability theorem in~\cite[Theorem~4.14]{GJPS:2020b} provides the combined inf--sup condition
\begin{equation}
\label{LBB}
\sup_{(\tau_h,q_h)\in\Sigma_h^+\times Q_h}
\frac{
b_1(v_h,q_h)+b_2(\tau_h,(v_h,\gamma_h))
}{
\|\tau_h\|_{\Sigma_h^+}+\|q_h\|_{Q_h}
}
\gtrsim
\|(v_h,\gamma_h)\|_{U_h},
\quad (v_h,\gamma_h)\in U_h.
\end{equation}
We will also use
\begin{equation}
\label{LBB_b1}
\sup_{v_h\in V_h}
\frac{b_1(v_h,q_h)}{\|v_h\|_{V_h}}
\gtrsim
\|q_h\|_{Q_h},
\quad q_h\in Q_h,
\end{equation}
whose proof is given in~\cref{App:LBB_b1}.

% Subsection: Algebraic formulation
\subsection{Algebraic formulation}
We henceforth identify finite element functions with their coefficient vectors. The discrete problem~\eqref{MCS_FEM} has the algebraic form
\begin{subequations}
\label{algebraic_formulation}
\begin{equation}
\mathcal{A}
\begin{bmatrix}
\sigma_h\\
(u_h,\omega_h)\\
p_h
\end{bmatrix}
:=
\begin{bmatrix}
A&B_2^T&0\\
B_2&0&B_1^T\\
0&B_1&0
\end{bmatrix}
\begin{bmatrix}
\sigma_h\\
(u_h,\omega_h)\\
p_h
\end{bmatrix}
=
\begin{bmatrix}
0\\g\\0
\end{bmatrix},
\end{equation}
where
\begin{equation}
\begin{aligned}
(\varsigma,\tau)_A&=a(\varsigma,\tau),
&&\quad \varsigma,\tau\in\Sigma_h^+, \\
(B_1(v,\gamma),q)&=b_1((v,\gamma),q),
&&\quad (v,\gamma)\in U_h,\ q\in Q_h, \\
(B_2\tau,(v,\gamma))&=b_2(\tau,(v,\gamma)),
&&\quad \tau\in\Sigma_h^+,\ (v,\gamma)\in U_h.
\end{aligned}
\end{equation}
\end{subequations}
Here $(\cdot,\cdot)$ is the Euclidean inner product and $g$ represents the functional $(v,\gamma)\mapsto-(f,v)$. Since $b_1$ acts only on the velocity component,
\[
B_1=\begin{bmatrix}B_{V_h}&0\end{bmatrix}
\]
with respect to $U_h=V_h\times W_h$.

For later use, define the mass matrices
\begin{equation}
\label{mass_matrices}
\begin{aligned}
(\varsigma,\tau)_{M_{\Sigma_h^+}}
&=(\varsigma,\tau)_{L^2(\Omega)},
&&\quad \varsigma,\tau\in\Sigma_h^+, \\
(v,w)_{M_{V_h}}
&=(v,w)_{V_h},
&&\quad v,w\in V_h, \\
(\gamma,\delta)_{M_{W_h}}
&=(\gamma,\delta)_{L^2(\Omega)},
&&\quad \gamma,\delta\in W_h, \\
((v,\gamma),(w,\delta))_{M_{U_h}}
&=((v,\gamma),(w,\delta))_{U_h},
&&\quad (v,\gamma),(w,\delta)\in U_h, \\
(q,r)_{M_{Q_h}}
&=(q,r)_{L^2(\Omega)},
&&\quad q,r\in Q_h.
\end{aligned}
\end{equation}
The inner products on $V_h$ and $U_h$ are those induced by~\eqref{norms}. With respect to $U_h=V_h\times W_h$,
\begin{equation}
\label{MUh}
M_{U_h}
=
\begin{bmatrix}
M_{V_h}+C_h^TM_{W_h}C_h&-C_h^TM_{W_h}\\
-M_{W_h}C_h&M_{W_h}
\end{bmatrix}
=
\begin{bmatrix}
I&-C_h^T\\
0&I
\end{bmatrix}
\begin{bmatrix}
M_{V_h}&0\\
0&M_{W_h}
\end{bmatrix}
\begin{bmatrix}
I&0\\
-C_h&I
\end{bmatrix},
\end{equation}
where
\begin{equation}
\label{Ch}
C_hv=\kappa(\operatorname{curl}v),
\quad v\in V_h.
\end{equation}
Moreover,~\eqref{coercivity_A} gives
\[
(\varsigma,\tau)_A
=
a(\varsigma,\tau)
=
\nu^{-1}(\varsigma,\tau)_{L^2(\Omega)}
=
\nu^{-1}(M_{\Sigma_h^+}\varsigma,\tau),
\quad \varsigma,\tau\in\Sigma_h^+,
\]
and hence
\begin{equation}
\label{A}
A=\nu^{-1}M_{\Sigma_h^+}.
\end{equation}

The Riesz representation of~\eqref{LBB} gives the following equivalent algebraic statement: for every $(v,\gamma)\in U_h$, there exist $\tau\in\Sigma_h^+$ and $q\in Q_h$ such that
\begin{equation}
\label{LBB_algebraic}
B_2\tau+B_1^Tq=M_{U_h}(v,\gamma),
\quad
((v,\gamma),(v,\gamma))_{M_{U_h}}
\gtrsim
(\tau,\tau)_{M_{\Sigma_h^+}}+(q,q)_{M_{Q_h}}.
\end{equation}

For $\eta>0$, we use the augmented Lagrangian formulation
\begin{equation}
\label{augmented_formulation}
\mathcal{A}_\eta
\begin{bmatrix}
\sigma_h\\(u_h,\omega_h)\\p_h
\end{bmatrix}
:=
\begin{bmatrix}
A&B_2^T&0\\
B_2&-\eta\nu B_1^TM_{Q_h}^{-1}B_1&B_1^T\\
0&B_1&0
\end{bmatrix}
\begin{bmatrix}
\sigma_h\\(u_h,\omega_h)\\p_h
\end{bmatrix}
=
\begin{bmatrix}
0\\g\\0
\end{bmatrix}.
\end{equation}
The third block equation gives $B_1(u_h,\omega_h)=0$, so the augmentation vanishes at a solution. Consequently,~\eqref{algebraic_formulation} and~\eqref{augmented_formulation} have the same solution.

The Schur complements associated with~\eqref{augmented_formulation} are
\begin{equation}
\label{Schur_complements}
\begin{aligned}
S_{U_h,\eta}
&=B_2A^{-1}B_2^T+\eta\nu B_1^TM_{Q_h}^{-1}B_1
\stackrel{\eqref{A}}{=}
\nu (B_2M_{\Sigma_h^+}^{-1}B_2^T+\eta B_1^TM_{Q_h}^{-1}B_1 ), \\
S_{Q_h,\eta}
&=B_1S_{U_h,\eta}^{-1}B_1^T.
\end{aligned}
\end{equation}
The augmentation makes $S_{U_h,\eta}$ SPD through~\eqref{LBB}; without it, $b_2$ alone does not satisfy the required inf--sup condition.

% Remark: Matrix and operator representations
\begin{remark}
\label{Rem:matrix_operator}
The matrix blocks in~\eqref{algebraic_formulation} are assembled with respect to the Euclidean inner product of coefficient vectors.
For each discrete space $X_h\in\{\Sigma_h^+,U_h,Q_h\}$, equipped with the inner product used in~\eqref{mass_matrices}, let $\mathcal{B}_X=\{\phi_j\}_{j=1}^{N_X}$ denote the chosen basis and let $\mathcal{B}_X^*=\{\phi_j^*\}_{j=1}^{N_X}$ denote its dual basis, defined by $(\phi_j^*,\phi_l)_{X_h}=\delta_{jl}$.
For a linear operator $T\colon X_h\to Y_h$ and bases $\mathcal{B}$ of $X_h$ and $\mathcal{B}'$ of $Y_h$, we write $[T]_{\mathcal{B}\to\mathcal{B}'}$ for the corresponding matrix representation; see~\cite[Section~3]{XZ:2017}.
With this notation, $M_{X_h}=[I_{X_h}]_{\mathcal{B}_X\to\mathcal{B}_X^*}$, and every assembled block maps coefficients in the primal basis of its trial space to coefficients in the dual basis of its test space.

For instance, let $D_h\colon U_h\to Q_h$ denote the discrete divergence, defined by
\[
(D_h(v,\gamma),q)_{L^2(\Omega)}
=
b_1((v,\gamma),q),
\quad (v,\gamma)\in U_h,\ q\in Q_h.
\]
Then
\[
B_1=[D_h]_{\mathcal{B}_U\to\mathcal{B}_Q^*},
\quad
M_{Q_h}^{-1}B_1=[D_h]_{\mathcal{B}_U\to\mathcal{B}_Q}.
\]
Thus, $B_1$ and $M_{Q_h}^{-1}B_1$ represent the same operator with output coefficients in the dual and primal bases of $Q_h$, respectively. The same convention applies to $A$ and $B_2$.
\end{remark}

% Section: Block diagonal preconditioners
\section{Block diagonal preconditioners}
\label{Sec:Diagonal}
In this section, we present optimal block diagonal preconditioners for the augmented Lagrangian formulation~\eqref{augmented_formulation}.

% Subsection: Preconditioner with Schur complements
\subsection{Preconditioner with Schur complements}
The starting point is the following block diagonal preconditioner defined in terms of the Schur complements~\eqref{Schur_complements}:
\begin{equation}
\label{preconditioner_diagonal_exact}
\mathcal{L}_{\eta, D}
=
\begin{bmatrix}
A^{-1} & 0 & 0 \\
0 & S_{U_h, \eta}^{-1} & 0 \\
0 & 0 & S_{Q_h, \eta}^{-1}
\end{bmatrix}.
\end{equation}
In \cref{Thm:preconditioner_diagonal_exact}, we establish an optimal preconditioning property of~\eqref{preconditioner_diagonal_exact}, which is analogous to the classical result on block preconditioners for saddle point systems proposed in~\cite{MGW:2000}.
Similar results on block diagonal preconditioners for double saddle point systems can be found in~\cite[Theorem~3.1]{AB:2018} and~\cite[Theorem~2.1]{HM:2019}.

% Theorem: Exact block diagonal preconditioner
\begin{theorem}
\label{Thm:preconditioner_diagonal_exact}
Let $\mathcal{L}_{\eta,D}$ be the block diagonal preconditioner presented in~\eqref{preconditioner_diagonal_exact} for the augmented Lagrangian formulation~\eqref{augmented_formulation}.
Then the preconditioned operator $\mathcal{L}_{\eta, D} \mathcal{A}_{\eta}$ satisfies
\[
\sigma(\mathcal{L}_{\eta,D}\mathcal{A}_\eta)
\subset
[-\psi,-\psi^{-1}] \cup [\psi^{-1},\psi],
\quad
\psi=2(1+\sqrt5).
\]
\end{theorem}
\begin{proof}
Since $\mathcal{L}_{\eta, D}\mathcal{A}_{\eta}$ is similar to $\mathcal{L}_{\eta, D}^{1/2}\mathcal{A}_{\eta}\mathcal{L}_{\eta, D}^{1/2}$, it is enough to study the latter. Define
\begin{equation*}
\hat{B}_1 = S_{Q_h, \eta}^{-\frac{1}{2}} B_1 S_{U_h, \eta}^{-\frac{1}{2}},
\quad
\hat{B}_2 = S_{U_h, \eta}^{-\frac{1}{2}} B_2 A^{-\frac{1}{2}}.
\end{equation*}
It follows from~\eqref{Schur_complements} that
\begin{equation}
\label{normalized_blocks_diagonal}
\hat{B}_1\hat{B}_1^T=I,
\quad
0\leq \hat{B}_2\hat{B}_2^T
=I-\eta\nu S_{U_h,\eta}^{-\frac12}B_1^TM_{Q_h}^{-1}B_1S_{U_h,\eta}^{-\frac12}
\leq I.
\end{equation}
A direct calculation gives
\begin{equation}
\label{factorization_diagonal}
\mathcal{L}_{\eta, D}^{\frac{1}{2}} \mathcal{A}_{\eta} \mathcal{L}_{\eta, D}^{\frac{1}{2}}
=
\mathcal{E}
\begin{bmatrix}
I & 0 \\
0 & \mathcal{H}
\end{bmatrix}
\mathcal{E}^T,
\end{equation}
where
\begin{equation*}
\mathcal{E}
=
\begin{bmatrix}
I & 0 & 0 \\
\hat{B}_2 & I & 0 \\
0 & 0 & I
\end{bmatrix},
\quad
\mathcal{H}
=
\begin{bmatrix}
-I & \hat{B}_1^T \\
\hat{B}_1 & 0
\end{bmatrix}.
\end{equation*}
By~\eqref{normalized_blocks_diagonal},
\begin{equation*}
\|\mathcal{E}\|\leq 1+\|\hat{B}_2\|\leq 2,
\quad
\|\mathcal{E}^{-1}\|\leq 1+\|\hat{B}_2\|\leq 2.
\end{equation*}
Moreover, the eigenvalue equation for $\mathcal{H}$ and the identity $\hat{B}_1\hat{B}_1^T=I$ show that either $\lambda=-1$ or $\lambda^2+\lambda-1=0$. Hence
\begin{equation*}
\sigma(\mathcal{H})
\subset
\left\{-1,\frac{-1-\sqrt5}{2},\frac{-1+\sqrt5}{2}\right\}.
\end{equation*}
Combining these estimates with~\eqref{factorization_diagonal} and the congruence estimate for symmetric matrices in~\cite[Corollary~4.2]{Park:2025}, we obtain
\begin{equation*}
\resizebox{\textwidth}{!}{$\displaystyle
\max |\sigma (\mathcal{L}_{\eta, D}^{\frac{1}{2}} \mathcal{A}_{\eta} \mathcal{L}_{\eta, D}^{\frac{1}{2}} ) |
\leq 2^2\frac{1+\sqrt5}{2}=\psi, \quad
\min |\sigma (\mathcal{L}_{\eta, D}^{\frac{1}{2}} \mathcal{A}_{\eta} \mathcal{L}_{\eta, D}^{\frac{1}{2}} ) |
\geq 2^{-2}\frac{-1+\sqrt5}{2}=\psi^{-1}.
$}
\end{equation*}
This completes the proof.
\end{proof}

% Subsection: Schur complement estimates and practical realization
\subsection{Schur complement estimates and practical realization}
We next replace the Schur complements in~\eqref{preconditioner_diagonal_exact} by spectrally equivalent matrices that can be inverted efficiently.
To handle $S_{U_h,\eta}$ and $S_{Q_h,\eta}$, we use the following result from~\cite[Theorem~4.3]{Park:2025} on the spectrum of preconditioned Schur complements.

% Lemma: Spectrum of the Schur complement
\begin{lemma}
\label{Lem:Schur}
Let $A \colon V \to V$ and $L \colon W \to W$ be SPD linear operators, and let $B \colon V \to W$ be a surjective linear operator.
Then the Schur complement $S = B A^{-1} B^T$ satisfies
\begin{equation*}
    \lambda_{\min} (L S)
    = \inf_{q \in W} \sup_{\substack{v \in V \\ Bv = q}} \frac{\|q\|_L^2}{\|v\|_A^2},
    \quad
    \lambda_{\max} (L S)
    = \sup_{v \in V} \frac{\|Bv\|_L^2}{\|v\|_A^2}.
\end{equation*}
\end{lemma}

In \cref{Prop:Schur_complements_mass}, we show that the matrices $S_{U_h,\eta}$ and $S_{Q_h, \eta}$ are spectrally equivalent to the mass matrices $M_{U_h}$ and $M_{Q_h}$ defined in~\eqref{mass_matrices}, up to a scaling factor, respectively.

% Proposition: Estimates for the Schur complements
\begin{proposition}
\label{Prop:Schur_complements_mass}
The Schur complements defined in~\eqref{Schur_complements} satisfy
\begin{equation}
\label{Schur_complements_mass}
S_{U_h,\eta}\eqsim \nu M_{U_h},
\quad
S_{Q_h,\eta}\eqsim \nu^{-1}M_{Q_h}.
\end{equation}
\end{proposition}
\begin{proof}
By~\eqref{A},
\begin{equation}
\label{SUh_factorization}
\nu^{-1} S_{U_h, \eta}
=
\begin{bmatrix} B_2 & B_1^T \end{bmatrix}
\begin{bmatrix}
M_{\Sigma_h^+} & 0 \\
0 & \eta^{-1} M_{Q_h}
\end{bmatrix}^{-1}
\begin{bmatrix}
B_2^T \\
B_1
\end{bmatrix}.
\end{equation}
Since $M_{U_h}$ is invertible, the algebraic inf--sup condition~\eqref{LBB_algebraic} implies that
\[
\begin{bmatrix} B_2 & B_1^T \end{bmatrix}
\colon \Sigma_h^+\times Q_h\to U_h
\]
is surjective. Hence, \cref{Lem:Schur} applies to~\eqref{SUh_factorization}.
For the minimum eigenvalue, by \cref{Lem:Schur}, \eqref{SUh_factorization}, and~\eqref{LBB_algebraic}, we have
\begin{equation}
\label{SUh_lower_bound}
\begin{aligned}
\lambda_{\min}(\nu^{-1}M_{U_h}^{-1}S_{U_h,\eta})
&=
\inf_{y\in U_h}
\sup_{\substack{\tau\in\Sigma_h^+,\ r\in Q_h,\\
B_2\tau+B_1^Tr=M_{U_h}y}}
\frac{\|y\|_{M_{U_h}}^2}
{\|\tau\|_{M_{\Sigma_h^+}}^2+\eta^{-1}\|r\|_{M_{Q_h}}^2}
\gtrsim 1.
\end{aligned}
\end{equation}
For the upper bound, the continuity estimates~\eqref{continuity_B_1} and~\eqref{continuity_B_2}, together with~\eqref{norm_V_h_min}, imply
\begin{align}
\label{B2_mass_bound}
\|B_2\tau\|_{M_{U_h}^{-1}}^2
&=
\sup_{z\in U_h}
\frac{b_2(\tau,z)^2}{\|z\|_{U_h}^2}
\lesssim
\|\tau\|_{M_{\Sigma_h^+}}^2, \\
\label{B1_mass_bound}
\|B_1^Tq\|_{M_{U_h}^{-1}}^2
&=
\sup_{(v,\gamma)\in U_h}
\frac{b_1(v,q)^2}{\|(v,\gamma)\|_{U_h}^2}
=
\sup_{v\in V_h}
\frac{b_1(v,q)^2}{\|v\|_{V_h}^2}
\lesssim
\|q\|_{M_{Q_h}}^2.
\end{align}
Therefore, by \cref{Lem:Schur}, \eqref{SUh_factorization}, \eqref{B2_mass_bound}, and~\eqref{B1_mass_bound},
\begin{equation*}
\lambda_{\max}(\nu^{-1}M_{U_h}^{-1}S_{U_h,\eta})
=
\sup_{\tau\in\Sigma_h^+,\ q\in Q_h}
\frac{\|B_2\tau+B_1^Tq\|_{M_{U_h}^{-1}}^2}
{\|\tau\|_{M_{\Sigma_h^+}}^2+\eta^{-1}\|q\|_{M_{Q_h}}^2}
\lesssim 1.
\end{equation*}
This proves the first equivalence in~\eqref{Schur_complements_mass}.

For the second equivalence, the first one gives
\begin{equation*}
S_{Q_h,\eta}
=B_1S_{U_h,\eta}^{-1}B_1^T
\eqsim
\nu^{-1}B_1M_{U_h}^{-1}B_1^T.
\end{equation*}
Moreover, the continuity estimate~\eqref{continuity_B_1}, the discrete inf--sup condition~\eqref{LBB_b1}, and \eqref{norm_V_h_min} yield
\begin{equation*}
\|q\|_{B_1M_{U_h}^{-1}B_1^T}^2
=
\sup_{(v,\gamma)\in U_h}
\frac{b_1(v,q)^2}{\|(v,\gamma)\|_{U_h}^2}
=
\sup_{v\in V_h}
\frac{b_1(v,q)^2}{\|v\|_{V_h}^2}
\eqsim
\|q\|_{M_{Q_h}}^2.
\end{equation*}
This completes the proof.
\end{proof}

The lower bound~\eqref{SUh_lower_bound} also provides insight into the choice of the augmented Lagrangian parameter $\eta$. Ideally, $\eta$ should be chosen in accordance with the hidden constants in the discrete inf--sup condition~\eqref{LBB_algebraic}. Since these constants are not available a priori, $\eta$ is selected by tuning; its effect is examined in~\cref{Sec:Numerical}.

% Preconditioner with mass matrices
Thanks to~\eqref{A} and~\cref{Prop:Schur_complements_mass}, the ideal block diagonal preconditioner~\eqref{preconditioner_diagonal_exact} satisfies
\begin{equation}
\label{preconditioner_diagonal_mass}
\mathcal{L}_{\eta,D}
\eqsim
\begin{bmatrix}
\nu M_{\Sigma_h^+}^{-1} & 0 & 0 \\
0 & \nu^{-1}M_{U_h}^{-1} & 0 \\
0 & 0 & \nu M_{Q_h}^{-1}
\end{bmatrix}
=:
\tilde{\mathcal{L}}_{\eta,D}.
\end{equation}
Combining~\cref{Thm:preconditioner_diagonal_exact} with the spectral equivalence established in~\eqref{preconditioner_diagonal_mass}, we obtain the following optimal conditioning property of $\tilde{\mathcal{L}}_{\eta,D}\mathcal{A}_\eta$.

% Theorem: Mass block diagonal preconditioner
\begin{theorem}
\label{Thm:preconditioner_diagonal_mass}
Let $\tilde{\mathcal{L}}_{\eta,D}$ be the block diagonal preconditioner presented in~\eqref{preconditioner_diagonal_mass} for the augmented Lagrangian formulation~\eqref{augmented_formulation}.
Then the preconditioned operator $\tilde{\mathcal{L}}_{\eta,D}\mathcal{A}_\eta$ satisfies
\[
\sigma(\tilde{\mathcal{L}}_{\eta,D}\mathcal{A}_\eta)
\subset
[-c_{\max},-c_{\min}]\cup[c_{\min},c_{\max}],
\]
for some positive constants $c_{\min}$ and $c_{\max}$ independent of $\nu$ and $h$.
\end{theorem}

% Preconditioner with low computational cost
We now replace the inverse mass matrices in~\eqref{preconditioner_diagonal_mass} by inexpensive spectrally equivalent approximations.

Since $M_{\Sigma_h^+}$ is an $L^2$-mass matrix, it admits a spectrally equivalent diagonal lumped version $\hat{M}_{\Sigma_h^+}$~\cite{VSB:2023,Wathen:1987}. We set
\begin{equation}
\label{LSigmah}
L_{\Sigma_h^+}
=
\hat{M}_{\Sigma_h^+}^{-1}
\eqsim
M_{\Sigma_h^+}^{-1}.
\end{equation}

To account for the zero-mean constraint in $Q_h$, let $M_{\mathbb{P}^k(\mathcal{T}_h)}$ be the $L^2$-mass matrix on $\mathbb{P}^k(\mathcal{T}_h)$ and let $P_0\colon\mathbb{P}^k(\mathcal{T}_h)\to Q_h$ be the $L^2$-orthogonal projection. Then
\begin{equation}
\label{LQh}
M_{Q_h}^{-1}
=
P_0M_{\mathbb{P}^k(\mathcal{T}_h)}^{-1}P_0^T,
\quad
L_{Q_h}
=
P_0\hat{M}_{\mathbb{P}^k(\mathcal{T}_h)}^{-1}P_0^T
\eqsim
M_{Q_h}^{-1}.
\end{equation}
Here $M_{\mathbb{P}^k(\mathcal{T}_h)}$ is elementwise block diagonal, while applying $P_0$ amounts to a rank-one correction. Thus, one application of $L_{Q_h}$ consists of a diagonal scaling and a rank-one correction.

For $M_{U_h}$, assume that an SPD preconditioner $L_{V_h}$ satisfies
\begin{equation}
\label{LVh}
L_{V_h}\eqsim M_{V_h}^{-1}.
\end{equation}
An example will be presented in the next subsection. Using the factorization~\eqref{MUh}, we define
\begin{equation}
\label{LUh}
L_{U_h}
=
\begin{bmatrix}
I & 0 \\
C_h & I
\end{bmatrix}
\begin{bmatrix}
L_{V_h} & 0 \\
0 & M_{W_h}^{-1}
\end{bmatrix}
\begin{bmatrix}
I & C_h^T \\
0 & I
\end{bmatrix}.
\end{equation}
Since $W_h$ consists of elementwise polynomials, $M_{W_h}$ is elementwise block diagonal. Moreover, a direct calculation using~\eqref{MUh} and~\eqref{LUh} gives
\begin{equation}
\label{LUh_similarity}
L_{U_h}M_{U_h}
=
\begin{bmatrix}
I & 0 \\
C_h & I
\end{bmatrix}
\begin{bmatrix}
L_{V_h}M_{V_h} & 0 \\
0 & I
\end{bmatrix}
\begin{bmatrix}
I & 0 \\
C_h & I
\end{bmatrix}^{-1}.
\end{equation}
Hence
\begin{equation}
\label{LUh_equivalence}
\sigma(L_{U_h}M_{U_h})
=
\sigma(L_{V_h}M_{V_h})\cup\{1\},
\quad
L_{U_h}\eqsim M_{U_h}^{-1}.
\end{equation}

Based on~\eqref{LSigmah},~\eqref{LQh},~\eqref{LUh_equivalence}, and~\cref{Thm:preconditioner_diagonal_mass}, we present the final block diagonal preconditioner
\begin{equation}
\label{preconditioner_diagonal_low_cost}
\hat{\mathcal{L}}_{\eta,D}
=
\begin{bmatrix}
\nu L_{\Sigma_h^+} & 0 & 0 \\
0 & \nu^{-1}L_{U_h} & 0 \\
0 & 0 & \nu L_{Q_h}
\end{bmatrix}.
\end{equation}
The following optimal conditioning property of $\hat{\mathcal{L}}_{\eta,D}\mathcal{A}_\eta$ is immediate.

% Theorem: Final block diagonal preconditioner
\begin{theorem}
\label{Thm:preconditioner_diagonal_low_cost}
Let $\hat{\mathcal{L}}_{\eta,D}$ be the block diagonal preconditioner presented in~\eqref{preconditioner_diagonal_low_cost} for the augmented Lagrangian formulation~\eqref{augmented_formulation}.
Then the preconditioned operator $\hat{\mathcal{L}}_{\eta,D}\mathcal{A}_\eta$ satisfies
\[
\sigma(\hat{\mathcal{L}}_{\eta,D}\mathcal{A}_\eta)
\subset
[-c_{\max},-c_{\min}]\cup[c_{\min},c_{\max}],
\]
for some positive constants $c_{\min}$ and $c_{\max}$ independent of $\nu$ and $h$.
\end{theorem}

% Remark: Zero eta
\begin{remark}
\label{Rem:eta_zero}
The preconditioners~\eqref{preconditioner_diagonal_mass} and~\eqref{preconditioner_diagonal_low_cost} are well-defined for $\eta=0$ and can therefore also be applied to the non-augmented system. The assumption $\eta>0$ enters the analysis through the use of the discrete inf--sup condition~\eqref{LBB_algebraic} in the proof of~\cref{Prop:Schur_complements_mass}. Numerical results for $\eta=0$ are presented in~\cref{Sec:Numerical}.
\end{remark}

% Subsection: Construction of L_{V_h}
\subsection{Construction of \texorpdfstring{$L_{V_h}$}{LVh}}
Here, we present an example of a preconditioner $L_{V_h}$ for the matrix $M_{V_h}$ satisfying~\eqref{LVh}.
More precisely, we present an auxiliary space preconditioner~\cite{PX:2025,Xu:1996} for $M_{V_h}$.
Recall from~\eqref{norms} and~\eqref{mass_matrices} that
\[
(v,w)_{M_{V_h}}
=
(v,w)_{V_h}
=
\sum_{T \in \mathcal{T}_h} \int_T \varepsilon(v):\varepsilon(w)\,dx
+
\sum_{F \in \mathcal{F}_h} h^{-1}\int_F \llbracket v_t \rrbracket \cdot \llbracket w_t \rrbracket\,ds.
\]
Hence, $M_{V_h}$ is the stiffness matrix of a nonconforming discretization of the symmetric gradient energy.
This observation provides a natural way to precondition $M_{V_h}$, namely, to use an auxiliary conforming $H_0^1$-space.
We define
\begin{equation}
\label{Vh_tilde}
\tilde{V}_h
:=
H_0^1(\Omega, \mathbb{R}^d)\cap \mathbb{P}^1(\mathcal{T}_h,\mathbb{R}^d).
\end{equation}
Then we have $\tilde{V}_h \subset V_h$.
Let $I_h \colon \tilde{V}_h \to V_h$ denote the natural inclusion operator.

On $\tilde{V}_h$ we consider the conforming bilinear form $( \cdot, \cdot)_{\tilde{V}_h}$ corresponding to the symmetric gradient energy and its associated matrix $M_{\tilde{V}_h}$:
\begin{equation}
\label{M_Vh_tilde}
(\tilde{v},\tilde{w})_{M_{\tilde{V}_h}}
=
(\tilde{v},\tilde{w})_{\tilde{V}_h}
=
\int_\Omega \varepsilon(\tilde{v}):\varepsilon(\tilde{w})\,dx,
\quad
\tilde{v},\tilde{w} \in \tilde{V}_h.
\end{equation}
Since $\tilde{V}_h \subset H_0^1(\Omega, \mathbb{R}^d)$, the Korn inequality yields
\begin{equation*}
\|\tilde{v}\|_{M_{\tilde{V}_h}}^2
\simeq
\sum_{i=1}^d | \tilde{v}_i |_{H^1(\Omega)}^2,
\quad
\tilde{v} = (\tilde{v}_1, \dots, \tilde{v}_d) \in \tilde{V}_h.
\end{equation*}
Thus, any optimal preconditioner $L_{\tilde{V}_h}$ satisfying
\begin{equation}
\label{LVh_tilde}
L_{\tilde{V}_h} \eqsim M_{\tilde{V}_h}^{-1}
\end{equation}
for the vector Laplacian, or equivalently for conforming linear elasticity, may be used to precondition $M_{\tilde{V}_h}$; for example, Bramble--Pasciak--Xu-type preconditioners~\cite{BPX:1990,JPX:2025,LX:2016} or algebraic multigrid preconditioners~\cite{XZ:2017}.

Following the idea in~\cite{LX:2016,Xu:1996}, we define $L_{V_h}$ by
\begin{equation}
\label{LVh_aux}
L_{V_h} = D_{V_h}^{-1} + I_h L_{\tilde{V}_h} I_h^T,
\end{equation}
where $D_{V_h}$ is the diagonal of $M_{V_h}$ and $L_{\tilde{V}_h}$ is chosen to satisfy~\eqref{LVh_tilde}.
Then we have the following optimal preconditioning property.
A proof of \cref{Lem:LVh} is provided in \cref{App:LVh}.

% Lemma: L_{V_h}
\begin{lemma}
\label{Lem:LVh}
Suppose that~\eqref{LVh_tilde} holds.
Then the preconditioner $L_{V_h}$ defined in~\eqref{LVh_aux} and the matrix $M_{V_h}$ defined in~\eqref{mass_matrices} satisfy~\eqref{LVh}.
\end{lemma}

% Section: Block triangular preconditioners
\section{Block triangular preconditioners}
\label{Sec:Triangular}
In this section, we present block lower triangular preconditioners for the augmented Lagrangian formulation~\eqref{augmented_formulation}.
Since these preconditioners are nonsymmetric, the natural outer iteration is GMRES~\cite{SS:1986}, and the analysis is carried out by means of FOV estimates~\cite{LW:2004,EE:2001,Starke:1997}.

For this purpose, it is convenient to work with the following equivalent formulation of~\eqref{augmented_formulation}, obtained by multiplying the second block row by $-1$:
\begin{equation}
\label{triangular_equivalent_system}
\bar{\mathcal{A}}_{\eta}
\begin{bmatrix}
\sigma_h \\
(u_h,\omega_h) \\
p_h
\end{bmatrix}
:=
\begin{bmatrix}
A & B_2^T & 0 \\
-B_2 & \eta \nu B_1^T M_{Q_h}^{-1} B_1 & -B_1^T \\
0 & B_1 & 0
\end{bmatrix}
\begin{bmatrix}
\sigma_h \\
(u_h,\omega_h) \\
p_h
\end{bmatrix}
=
\begin{bmatrix}
0 \\
-g \\
0
\end{bmatrix}.
\end{equation}
Clearly,~\eqref{triangular_equivalent_system} and~\eqref{augmented_formulation} have the same solution.

In~\cite{LW:2004}, it was shown that preconditioned GMRES with preconditioner $\mathcal{L}$ for the matrix $\bar{\mathcal{A}}_{\eta}$ achieves an optimal convergence rate if the following two FOV conditions hold:
\begin{equation}
\label{FOV}
\inf_{\xi \in \Sigma_h^+ \times U_h \times Q_h} \frac{(\xi, \mathcal{L} \bar{\mathcal{A}}_{\eta} \xi)_{\mathcal{H}}}{(\xi, \xi)_{\mathcal{H}}} \gtrsim 1, \quad
\sup_{\xi \in \Sigma_h^+ \times U_h \times Q_h} \frac{\| \mathcal{L} \bar{\mathcal{A}}_{\eta} \xi \|_{\mathcal{H}}}{\| \xi \|_{\mathcal{H}}} \lesssim 1
\end{equation}
where $\mathcal{H}$ is an SPD matrix of the same size as $\bar{\mathcal{A}}_{\eta}$.

% Subsection: Preconditioner with Schur complements
\subsection{Preconditioner with Schur complements}
We first consider the following preconditioner, which is a block lower triangular counterpart of~\eqref{preconditioner_diagonal_exact}:
\begin{equation}
\label{preconditioner_triangular_exact}
\mathcal{L}_{\eta,T}
=
\begin{bmatrix}
A & 0 & 0 \\
-B_2 & S_{U_h,\eta} & 0 \\
0 & B_1 & S_{Q_h,\eta}
\end{bmatrix}^{-1},
\end{equation}
where the Schur complements $S_{U_h,\eta}$ and $S_{Q_h,\eta}$ are given in~\eqref{Schur_complements}.

In \cref{Thm:preconditioner_triangular_exact}, we present the optimal convergence property of the preconditioner $\mathcal{L}_{\eta,T}$ in view of~\eqref{FOV}.

% Theorem: Block triangular preconditioner with Schur
\begin{theorem}
\label{Thm:preconditioner_triangular_exact}
Let $\mathcal{L}_{\eta,T}$ be the block lower triangular preconditioner defined by~\eqref{preconditioner_triangular_exact}.
Then the preconditioned operator $\mathcal{L}_{\eta,T} \bar{\mathcal{A}}_{\eta}$ satisfies the FOV conditions~\eqref{FOV} with
\begin{equation*}
\mathcal{H}
=
\operatorname{diag}(A, S_{U_h,\eta}, S_{Q_h, \eta}).
\end{equation*}
\end{theorem}
\begin{proof}
A direct calculation using~\eqref{Schur_complements} yields
\[
\mathcal{L}_{\eta,T}\bar{\mathcal A}_{\eta}
=
\begin{bmatrix}
I & A^{-1}B_2^T & 0 \\
0 & I & -S_{U_h,\eta}^{-1}B_1^T \\
0 & 0 & I
\end{bmatrix}.
\]
Take any $\xi = (\tau, z, q) \in \Sigma_h^+ \times U_h \times Q_h$.
It follows that
\begin{equation}
\label{Thm1:preconditioner_triangular_exact}
(\xi,\mathcal{L}_{\eta,T}\bar{\mathcal A}_{\eta} \xi )_{\mathcal{H}}
= \| \tau \|_A^2 + \| z \|_{S_{U_h, \eta}}^2 + \| q \|_{S_{Q_h, \eta}}^2 + (\tau, B_2^T z) - (z, B_1^T q).
\end{equation}
Moreover, by the Cauchy--Schwarz inequality, we have
\begin{equation}
\label{Thm2:preconditioner_triangular_exact}
\begin{aligned}
|(\tau, B_2^T z)|
&\leq \| \tau \|_A \, \| B_2^T z \|_{A^{-1}}
\stackrel{\eqref{Schur_complements}}{\leq} \| \tau \|_A \, \| z \|_{S_{U_h, \eta}}, \\
| (z, B_1^T q) |
&\leq \| z \|_{S_{U_h, \eta}} \, \|  B_1^T q \|_{S_{U_h, \eta}^{-1}}
= \| z \|_{S_{U_h, \eta}} \, \| q \|_{S_{Q_h, \eta}}.
\end{aligned}
\end{equation}
Applying the Young's inequality with
\[
ab\le \frac34 a^2+\frac13 b^2,
\quad
bc\le \frac13 b^2+\frac34 c^2,
\]
to~\eqref{Thm1:preconditioner_triangular_exact}, we obtain
\[
\begin{aligned}
(\xi, \mathcal{L}_{\eta,T}\bar{\mathcal A}_{\eta} \xi )_{\mathcal{H}}
&\stackrel{\eqref{Thm2:preconditioner_triangular_exact}}{\ge}
\| \tau \|_A^2 + \| z \|_{S_{U_h, \eta}}^2 + \| q \|_{S_{Q_h, \eta}}^2
- \| \tau \|_A \, \| z \|_{S_{U_h, \eta}}
- \| z \|_{S_{U_h, \eta}} \, \| q \|_{S_{Q_h, \eta}}
\\
&\ge
\frac14 \| \tau \|_A^2
+
\frac13 \| z \|_{S_{U_h, \eta}}^2
+
\frac14 \| q \|_{S_{Q_h, \eta}}^2
\\
&\ge
\frac14 \|\xi \|_{\mathcal{H}}^2,
\end{aligned}
\]
which gives the lower FOV bound in~\eqref{FOV}.

On the other hand, we have
\[
\begin{aligned}
\|\mathcal{L}_{\eta,T}\bar{\mathcal A}_{\eta} \xi\|_{\mathcal{H}}^2
&=
\|\tau+A^{-1}B_2^Tz\|_A^2
+
\|z-S_{U_h,\eta}^{-1}B_1^Tq\|_{S_{U_h,\eta}}^2
+
\|q\|_{S_{Q_h,\eta}}^2
\\
&\lesssim
\|\tau\|_A^2
+
\|A^{-1}B_2^Tz\|_A^2
+
\|z\|_{S_{U_h,\eta}}^2
+
\|S_{U_h,\eta}^{-1}B_1^Tq\|_{S_{U_h,\eta}}^2
+
\|q\|_{S_{Q_h,\eta}}^2
\\
&\stackrel{\eqref{Schur_complements}}{\lesssim}
\|\tau\|_A^2
+
\|z\|_{S_{U_h,\eta}}^2
+
\|q\|_{S_{Q_h,\eta}}^2
\\
&=
\| \xi \|_{\mathcal{H}}^2,
\end{aligned}
\]
which proves the upper FOV bound in~\eqref{FOV}.
\end{proof}

% Subsection: Preconditioner with generic diagonal
\subsection{Preconditioner with generic diagonal}
Next, to improve computational efficiency, we replace each diagonal block in~\eqref{preconditioner_triangular_exact} with a generic preconditioner, in the same spirit as~\eqref{preconditioner_diagonal_mass} and~\eqref{preconditioner_diagonal_low_cost}. Specifically, we define
\begin{equation}
\label{preconditioner_triangular_generic}
\hat{\mathcal{L}}_{\eta,T}
=
\begin{bmatrix}
H_{\Sigma_h^+} & 0 & 0 \\
-B_2 & H_{U_h} & 0 \\
0 & B_1 & H_{Q_h}
\end{bmatrix}^{-1},
\end{equation}
where $H_{\Sigma_h^+}$, $H_{U_h}$, and $H_{Q_h}$ are SPD matrices.

In \cref{Thm:preconditioner_triangular_generic}, we show that the preconditioner~\eqref{preconditioner_triangular_generic} is optimal, provided that $H_{\Sigma_h^+}^{-1}$, $H_{U_h}^{-1}$, and $H_{Q_h}^{-1}$ are effective preconditioners for $A$, $S_{U_h,\eta}$, and $S_{Q_h,\eta}$, respectively. For instance, as discussed in \cref{Sec:Diagonal}, one may choose
\begin{equation}
\label{H_mass}
H_{\Sigma_h^+}^{-1} = \nu M_{\Sigma_h^+}^{-1},
\quad
H_{U_h}^{-1} = \nu^{-1} M_{U_h}^{-1},
\quad
H_{Q_h}^{-1} = \nu M_{Q_h}^{-1},
\end{equation}
or
\[
H_{\Sigma_h^+}^{-1} = \nu L_{\Sigma_h^+},
\quad
H_{U_h}^{-1} = \nu^{-1} L_{U_h},
\quad
H_{Q_h}^{-1} = \nu L_{Q_h}.
\]
Note that an application of the preconditioner $\hat{\mathcal{L}}_{\eta,T}$ in~\eqref{preconditioner_triangular_generic} does not require $H_{\Sigma_h^+}$, $H_{U_h}$, and $H_{Q_h}$, but only their inverses; see~\cite[Lemma~6.5]{PX:2025}.

% Theorem: Block triangular preconditioner with generic diagonal
\begin{theorem}
\label{Thm:preconditioner_triangular_generic}
Let $\hat{\mathcal{L}}_{\eta,T}$ be the block lower triangular preconditioner defined by~\eqref{preconditioner_triangular_generic}.
Assume that there exist constants $\rho_{\Sigma_h^+},\rho_{U_h} \in [0,1)$, independent of the problem parameter $\nu$ and the discretization parameter $h$,  such that
\begin{equation}
\label{triangular_generic_assumptions}
\|I_{\Sigma_h^+} - H_{\Sigma_h^+}^{-1} A\|_A
\le \rho_{\Sigma_h^+},
\quad
\|I_{U_h} - H_{U_h}^{-1} S_{U_h,\eta}\|_{S_{U_h,\eta}}
\le \rho_{U_h},
\quad
H_{Q_h} \eqsim S_{Q_h, \eta}.
\end{equation}
In addition, assume that
\begin{equation}
\label{triangular_generic_smallness}
1+\rho_{\Sigma_h^+} + (1+\rho_{U_h})(1+2\rho_{\Sigma_h^+}+2\rho_{\Sigma_h^+}^2) < 4.
\end{equation}
Then the preconditioned operator $\hat{\mathcal{L}}_{\eta,T}\bar{\mathcal{A}}_{\eta}$ satisfies the FOV conditions~\eqref{FOV} with
\begin{equation*}
\mathcal{H}
=
\operatorname{diag}
(
H_{\Sigma_h^+},
\,
H_{U_h},
\,
H_{Q_h}
).
\end{equation*}
\end{theorem}
\begin{proof}
We define
\[
\hat{S}_{U_h,\eta}
=
B_2 H_{\Sigma_h^+}^{-1} B_2^T
+
\eta \nu B_1^T M_{Q_h}^{-1} B_1.
\]
We observe that~\eqref{triangular_generic_assumptions} implies
\begin{equation}
\label{Thm1:preconditioner_triangular_generic}
\begin{aligned}
(1-\rho_{\Sigma_h^+}) H_{\Sigma_h^+}
&\le A \le (1+\rho_{\Sigma_h^+}) H_{\Sigma_h^+}, \\
(1-\rho_{U_h}) H_{U_h}
&\le S_{U_h,\eta} \le (1+\rho_{U_h}) H_{U_h}. \\
\end{aligned}
\end{equation}
Moreover, by~\eqref{Thm1:preconditioner_triangular_generic} and the definition of $\hat{S}_{U_h,\eta}$, we have
\begin{equation}
\label{Thm2:preconditioner_triangular_generic}
(1-\rho_{\Sigma_h^+}) S_{U_h,\eta}
\le
\hat{S}_{U_h,\eta}
\le
(1+\rho_{\Sigma_h^+}) S_{U_h,\eta}.
\end{equation}

A direct calculation yields
\[
\hat{\mathcal{L}}_{\eta,T}\bar{\mathcal A}_{\eta}
=
\begin{bmatrix}
H_{\Sigma_h^+}^{-1}A & H_{\Sigma_h^+}^{-1}B_2^T & 0 \\
- H_{U_h}^{-1} B_2(I-H_{\Sigma_h^+}^{-1}A) & H_{U_h}^{-1}\hat{S}_{U_h,\eta} & -H_{U_h}^{-1}B_1^T \\
H_{Q_h}^{-1}B_1H_{U_h}^{-1}B_2(I-H_{\Sigma_h^+}^{-1}A) &
H_{Q_h}^{-1}B_1(I-H_{U_h}^{-1}\hat{S}_{U_h,\eta}) &
H_{Q_h}^{-1}B_1H_{U_h}^{-1}B_1^T
\end{bmatrix}.
\]
Take any $\xi = (\tau, z, q) \in \Sigma_h^+ \times U_h \times Q_h$.
Then, after collecting terms, we obtain
\begin{multline}
\label{Thm3:preconditioner_triangular_generic}
(\xi,\hat{\mathcal{L}}_{\eta,T}\bar{\mathcal A}_{\eta} \xi )_{\mathcal{H}}
=
\| \tau \|_A^2
+
\| z \|_{\hat{S}_{U_h,\eta}}^2
+
\| q \|_{B_1 H_{U_h}^{-1} B_1^T}^2
\\
+
(B_2 H_{\Sigma_h^+}^{-1} A \tau, z)
+
(B_1 H_{U_h}^{-1} B_2 (I-H_{\Sigma_h^+}^{-1}A)\tau, q)
-
(B_1 H_{U_h}^{-1} \hat{S}_{U_h,\eta} z, q).
\end{multline}
We estimate the three cross terms in~\eqref{Thm3:preconditioner_triangular_generic}.
For the first term, by the Cauchy--Schwarz inequality,
\[
|(B_2 H_{\Sigma_h^+}^{-1} A \tau, z)|
\le
\|A\tau\|_{H_{\Sigma_h^+}^{-1}} \, \|B_2^T z\|_{H_{\Sigma_h^+}^{-1}}.
\]
Using~\eqref{Thm1:preconditioner_triangular_generic} and the definition of $\hat{S}_{U_h,\eta}$, we have
\[
\|A\tau\|_{H_{\Sigma_h^+}^{-1}}^2
=
(H_{\Sigma_h^+}^{-1}A\tau,\tau)_A
\le
(1+\rho_{\Sigma_h^+})\|\tau\|_A^2,
\]
and
\[
\|B_2^T z\|_{H_{\Sigma_h^+}^{-1}}^2
=
(B_2 H_{\Sigma_h^+}^{-1} B_2^T z, z)
\le
\|z\|_{\hat{S}_{U_h,\eta}}^2.
\]
Therefore, we deduce
\begin{equation}
\label{Thm4:preconditioner_triangular_generic}
|(B_2 H_{\Sigma_h^+}^{-1} A \tau, z)|
\le
\sqrt{1+\rho_{\Sigma_h^+}}\,
\|\tau\|_A
\,
\|z\|_{\hat{S}_{U_h,\eta}}.
\end{equation}
For the last term, again by the Cauchy--Schwarz inequality, we have
\[
|(B_1 H_{U_h}^{-1} \hat{S}_{U_h,\eta} z, q)|
\le
\|\hat{S}_{U_h,\eta} z\|_{H_{U_h}^{-1}}
\,
\|q\|_{B_1 H_{U_h}^{-1} B_1^T}.
\]
By~\eqref{Thm1:preconditioner_triangular_generic} and~\eqref{Thm2:preconditioner_triangular_generic},
\[
\|\hat{S}_{U_h,\eta} z\|_{H_{U_h}^{-1}}^2
\le
(1+\rho_{U_h})
\|\hat{S}_{U_h,\eta} z\|_{S_{U_h,\eta}^{-1}}^2
\le
(1+\rho_{U_h})(1+\rho_{\Sigma_h^+})
\|z\|_{\hat{S}_{U_h,\eta}}^2,
\]
and hence
\begin{equation}
\label{Thm5:preconditioner_triangular_generic}
|(B_1 H_{U_h}^{-1} \hat{S}_{U_h,\eta} z, q)|
\le
\sqrt{(1+\rho_{U_h})(1+\rho_{\Sigma_h^+})}\,
\|z\|_{\hat{S}_{U_h,\eta}}
\,
\|q\|_{B_1 H_{U_h}^{-1} B_1^T}.
\end{equation}
For the second term, we have
\[
|(B_1 H_{U_h}^{-1} B_2 (I-H_{\Sigma_h^+}^{-1}A)\tau, q)|
\le
\|B_2 (I-H_{\Sigma_h^+}^{-1}A)\tau\|_{H_{U_h}^{-1}}
\,
\|q\|_{B_1 H_{U_h}^{-1} B_1^T}.
\]
Note that, for any $\varsigma \in \Sigma_h^+$,
\begin{multline}
\label{B2_estimate}
\|B_2 \varsigma \|_{H_{U_h}^{-1}}
\stackrel{\eqref{Thm1:preconditioner_triangular_generic}}{\le}
\sqrt{1+\rho_{U_h}} \|B_2 \varsigma \|_{S_{U_h,\eta}^{-1}}
= \sqrt{1 + \rho_{U_h}} \sup_{z \in U_h}
\frac{(B_2 \varsigma, z)}{\|z \|_{S_{U_h,\eta}}} \\
\leq \sqrt{1+\rho_{U_h}} \| \varsigma \|_A \sup_{z \in U_h} \frac{\|B_2^T z\|_{A^{-1}}}{\|z\|_{S_{U_h,\eta}}}
\stackrel{\eqref{Schur_complements}}{\le}
\sqrt{1+\rho_{U_h}} \| \varsigma \|_A.
\end{multline}
Applying this with $\varsigma = (I-H_{\Sigma_h^+}^{-1}A)\tau$ and using~\eqref{triangular_generic_assumptions} and~\eqref{Thm1:preconditioner_triangular_generic}, we obtain
\begin{equation}
\label{Thm6:preconditioner_triangular_generic}
|(B_1 H_{U_h}^{-1} B_2 (I-H_{\Sigma_h^+}^{-1}A)\tau, q)|
\le
\rho_{\Sigma_h^+} \sqrt{1+\rho_{U_h}}\,
\|\tau\|_A
\,
\|q\|_{B_1 H_{U_h}^{-1} B_1^T}.
\end{equation}
Combining~\eqref{Thm3:preconditioner_triangular_generic}--\eqref{Thm6:preconditioner_triangular_generic}, we get
\begin{equation}
\resizebox{\textwidth}{!}{$\displaystyle
\label{Thm7:preconditioner_triangular_generic}
(\xi,\hat{\mathcal{L}}_{\eta,T}\bar{\mathcal A}_{\eta} \xi )_{\mathcal{H}}
\ge
\begin{bmatrix}
\|\tau\|_A \\ \|z\|_{\hat{S}_{U_h,\eta}} \\ \|q\|_{B_1 H_{U_h}^{-1} B_1^T}
\end{bmatrix}^T
\begin{bmatrix}
1 & - \frac{\sqrt{1+\rho_{\Sigma_h^+}}}{2} & -\frac{\rho_{\Sigma_h^+} \sqrt{1+\rho_{U_h}}}{2}\\
- \frac{\sqrt{1+\rho_{\Sigma_h^+}}}{2} & 1 & -\frac{\sqrt{(1+\rho_{U_h})(1+\rho_{\Sigma_h^+})}}{2} \\
-\frac{\rho_{\Sigma_h^+} \sqrt{1+\rho_{U_h}}}{2} & -\frac{\sqrt{(1+\rho_{U_h})(1+\rho_{\Sigma_h^+})}}{2} & 1
\end{bmatrix}
\begin{bmatrix}
\|\tau\|_A \\ \|z\|_{\hat{S}_{U_h,\eta}} \\ \|q\|_{B_1 H_{U_h}^{-1} B_1^T}
\end{bmatrix}.
$}
\end{equation}
By the Sylvester criterion~(see, e.g.,~\cite{Gilbert:1991}), $\rho_{\Sigma_h^+} < 1$ and~\eqref{triangular_generic_smallness} imply that the matrix in~\eqref{Thm7:preconditioner_triangular_generic} is SPD.
Since $\rho_{\Sigma_h^+} < 1$, the leading \(2\times 2\) principal minor is positive.
Hence, we have
\begin{equation*}
\resizebox{\textwidth}{!}{$\displaystyle
\begin{aligned}
&(\xi,\hat{\mathcal{L}}_{\eta,T}\bar{\mathcal A}_{\eta} \xi )_{\mathcal{H}}
\gtrsim
\|\tau\|_A^2
+
\|z\|_{\hat{S}_{U_h,\eta}}^2
+
\|q\|_{B_1 H_{U_h}^{-1} B_1^T}^2 \\
&\geq (1-\rho_{\Sigma_h^+})\|\tau\|_{H_{\Sigma_h^+}}^2 + (1-\rho_{\Sigma_h^+})(1-\rho_{U_h})\|z\|_{H_{U_h}}^2 + (1-\rho_{U_h}) \lambda_{\min} (H_{Q_h}^{-1} S_{Q_h, \eta}) \|q\|_{H_{Q_h}}^2 \\
&\gtrsim \| \xi \|_{\mathcal{H}}^2,
\end{aligned}
$}
\end{equation*}
where the second inequality follows from~\eqref{Thm1:preconditioner_triangular_generic} and~\eqref{Thm2:preconditioner_triangular_generic}.
This proves the lower FOV bound in~\eqref{FOV}.

It remains to prove the upper FOV bound.
Let
\begin{equation*}
    (\tau_*, z_*, q_*) = \hat{\mathcal{L}}_{\eta,T}\bar{\mathcal A}_{\eta}\xi.
\end{equation*}
We observe that
\begin{equation}
\label{Thm8:preconditioner_triangular_generic}
    q_* = H_{Q_h}^{-1} B_1 (z - z_*).
\end{equation}
Using the triangle inequality,~\eqref{Thm1:preconditioner_triangular_generic}, and~\eqref{Thm2:preconditioner_triangular_generic}, we obtain
\begin{equation}
\label{Thm9:preconditioner_triangular_generic}
\|\tau_*\|_{H_{\Sigma_h^+}}
\le
\|A\tau\|_{H_{\Sigma_h^+}^{-1}}
+
\|B_2^T z\|_{H_{\Sigma_h^+}^{-1}}
\lesssim
\|\tau\|_A
+
\|z\|_{\hat{S}_{U_h,\eta}}
\end{equation}
and
\begin{equation}
\label{Thm10:preconditioner_triangular_generic}
\begin{aligned}
\|z_*\|_{H_{U_h}}
&\le
\| B_2 (I - H_{\Sigma_h^+}^{-1} A) \tau \|_{H_{U_h}^{-1}} + \| \hat{S}_{U_h,\eta} z \|_{H_{U_h}^{-1}} + \| B_1^T q \|_{H_{U_h}^{-1}} \\
&\stackrel{\eqref{B2_estimate}}{\lesssim} \|\tau\|_A
+ \|z\|_{\hat{S}_{U_h,\eta}}
+ \|q\|_{B_1 H_{U_h}^{-1} B_1^T}.
\end{aligned}
\end{equation}
Furthermore, we have
\begin{multline}
\label{Thm11:preconditioner_triangular_generic}
\|q_*\|_{H_{Q_h}}
\stackrel{\eqref{Thm8:preconditioner_triangular_generic}}{=}
\|B_1(z-z_*)\|_{H_{Q_h}^{-1}}
\lesssim
\|B_1(z-z_*)\|_{S_{Q_h,\eta}^{-1}} \\
\lesssim \|z - z_*\|_{S_{U_h,\eta}}
\lesssim \|z - z_*\|_{\hat{S}_{U_h,\eta}}
\le \|z\|_{\hat{S}_{U_h,\eta}} + \|z_*\|_{\hat{S}_{U_h,\eta}},
\end{multline}
where the second inequality follows from the estimate
\[
\|B_1 w\|_{S_{Q_h,\eta}^{-1}}
=
\sup_{r\in Q_h}
\frac{(B_1 w,r)}{\|r\|_{S_{Q_h,\eta}}}
=
\sup_{r\in Q_h}
\frac{(w,B_1^T r)}{\|r\|_{S_{Q_h,\eta}}}
\le
\|w\|_{S_{U_h,\eta}},
\quad w \in U_h.
\]
Combining~\eqref{Thm9:preconditioner_triangular_generic}--\eqref{Thm11:preconditioner_triangular_generic} yields
\[
\|\hat{\mathcal{L}}_{\eta,T}\bar{\mathcal A}_{\eta}\xi\|_{\mathcal{H}}^2
\lesssim
\|\tau\|_A^2
+
\|z\|_{\hat{S}_{U_h,\eta}}^2
+
\|q\|_{B_1 H_{U_h}^{-1} B_1^T}^2.
\]
Finally, invoking~\eqref{Thm1:preconditioner_triangular_generic} and~\eqref{Thm2:preconditioner_triangular_generic} once again proves the upper FOV bound in~\eqref{FOV}.
\end{proof}

% Remark: Additional assumption
\begin{remark}
\label{Rem:preconditioner_triangular_generic}
The additional assumption~\eqref{triangular_generic_smallness} in \cref{Thm:preconditioner_triangular_generic} is reasonable, as in practice we can make $\rho_{\Sigma_h^+}$ and $\rho_{U_h}$ sufficiently small by performing one or several steps of preconditioned Richardson iterations at each block.
\end{remark}

% Section: Numerical results
\section{Numerical results}
\label{Sec:Numerical}
In this section, we present numerical results for the proposed preconditioners. All numerical experiments were implemented in NGSolve~\cite{Schoberl:2014} with the PETSc~\cite{PETSc} backend and executed on the Oxford ARC system running AlmaLinux 9, using a 48-core Cascade Lake CPU (Intel Xeon Platinum 8268 $@$ 2.90 GHz) with 384 GB of memory.

We consider the unit domain $\Omega = (0,1)^d$ for $d=2,3$, and use polynomial degree $k=1$ for the discrete spaces~\eqref{spaces}. The domain is partitioned into a uniform mesh of uniform squares in two dimensions or cubes in three dimensions, which is then further subdivided for the triangulation. The source term $f$ is chosen such that the exact solution $(u,p)$ of~\eqref{strong} is given by
\begin{align*}
u(x) &= \operatorname{curl}(x_1^2 (1-x_1)^2 x_2^2(1-x_2)^2), \\
p(x) &= x_1^5 + x_2^5 - \frac{1}{3}
\end{align*}
in two dimensions, and
\begin{align*}
u(x) &= \operatorname{curl}\left(x_1^2(1-x_1)^2 x_2^2(1-x_2)^2 x_3^2 (1-x_3)^2 (1,1,1)^T\right), \\
p(x) &= x_1^5 + x_2^5 + x_3^5 - \frac{1}{2}
\end{align*}
in three dimensions.

Regarding the proposed preconditioners, while we have considerable flexibility in their design, we specify one particular choice for brevity.
More precisely, for block diagonal preconditioning, we use MINRES~\cite{PS:1975} with the preconditioner $\tilde{\mathcal{L}}_{\eta,D}$ given in~\eqref{preconditioner_diagonal_mass}, equivalently $\hat{\mathcal{L}}_{\eta,D}$ given in~\eqref{preconditioner_diagonal_low_cost}, with
\begin{equation*}
L_{\Sigma_h^+} = M_{\Sigma_h^+}^{-1},
\quad
L_{U_h} = M_{U_h}^{-1},
\quad
L_{Q_h} = M_{Q_h}^{-1}.
\end{equation*}
For block triangular preconditioning, we use left preconditioned GMRES~\cite{SS:1986}, preconditioned by $\hat{\mathcal{L}}_{\eta,T}$ given in~\eqref{preconditioner_triangular_generic} with~\eqref{H_mass}.
Note that the choice~\eqref{H_mass} can be easily modified, if needed, to guarantee the assumption~\eqref{triangular_generic_assumptions} in view of \cref{Rem:preconditioner_triangular_generic}; however, here we use it as is in a heuristic manner.

In all numerical experiments, we use a stopping criterion such that the relative residual is less than $10^{-6}$.
The maximum number of iterations is set to 1{,}000.

% Subsection: Robustness with respect to $\nu$ and $h$
\subsection{Robustness with respect to \texorpdfstring{$\nu$ and $h$}{nu and h}}
For each dimension and viscosity, \cref{Table:robust} reports the minimum and maximum iteration counts over four successive mesh refinements. The two-dimensional problems range from 29{,}184 to 1{,}867{,}776 DOFs, and the three-dimensional problems from 4{,}449 to 2{,}220{,}096 DOFs.

% Table: Robustness with respect to $\nu$ and $h$
\begin{table}
    \centering
    \caption{Iteration counts for MINRES and GMRES preconditioned by the proposed block diagonal preconditioner~\eqref{preconditioner_diagonal_mass} and block triangular preconditioner~\eqref{preconditioner_triangular_generic}--\eqref{H_mass}, respectively, for different values of $\nu$ and mesh sizes $h$ in two and three dimensions. 
    Results are reported for both the non-augmented case~($\eta=0$) and the augmented case~($\eta=1$).}
    \scalebox{0.8}{
    \begin{tabular}{c|c|cccc}
    \toprule
    &
    \#DOFs &
    \makecell{Diagonal \\ \eqref{preconditioner_diagonal_mass} ($\eta = 0$)} &
    \makecell{Diagonal \\ \eqref{preconditioner_diagonal_mass} ($\eta = 1$)} & 
    \makecell{Triangular \\ \eqref{preconditioner_triangular_generic}--\eqref{H_mass} ($\eta = 0$)} & 
    \makecell{ Triangular \\ \eqref{preconditioner_triangular_generic}--\eqref{H_mass} ($\eta = 1$)} \\
    \hline
    \multirow{4}{*}{\makecell{2D \\ ($\nu = 1$)}}
    & 29,184 & 159 & 183 & 64 & 47 \\
    & 116,736 & 166 & 191 & 65 & 47 \\
    & 466,944 & 168 & 192 & 65 & 46\\
    & 1,867,776 & 169 & 190 & 64 & 45 \\
    
    \hline
    \multirow{4}{*}{\makecell{2D \\ ($\nu = 10^{-3}$)}}
    & 29,184 & 157 & 180 & 66 & 48 \\
    & 116,736 & 162 & 185 & 67 & 48 \\
    & 466,944 & 164 & 187 & 67 & 48 \\
    & 1,867,776 & 165 & 186 & 66 & 48 \\

    \hline
    \multirow{4}{*}{\makecell{2D \\ ($\nu = 10^{-6}$)}}
    & 29,184 & 157 & 180 & 51 & 37 \\
    & 116,736 & 162 & 185 & 51 & 37 \\
    & 466,944 & 164 & 187 & 50 & 37 \\
    & 1,867,776 & 165 & 186 & 49 & 36 \\

    \hline
    \multirow{4}{*}{\makecell{3D \\ ($\nu = 1$)}}
    & 4,449 & 371 &  426 & 122 & 96 \\
    & 35,076 & 426 & 487 & 158 & 119 \\
    & 278,544 & 444 & 506 & 186 & 138 \\
    & 2,220,096 & 435 & 490 & 187 & 138 \\

    \hline
    \multirow{4}{*}{\makecell{3D \\ ($\nu = 10^{-3}$)}}
    & 4,449 & 362 & 414 & 143 & 114 \\
    & 35,076 & 390 & 441 & 175 & 132 \\
    & 278,544 & 410 & 466 & 203 & 151 \\
    & 2,220,096 & 414 & 469 & 196 & 146 \\

    \hline
    \multirow{4}{*}{\makecell{3D \\ ($\nu = 10^{-6}$)}}
    & 4,449 & 362 & 415 & 113 & 90 \\
    & 35,076 & 390 & 441 & 130 & 99 \\
    & 278,544 & 410 & 466 & 147 & 109 \\
    & 2,220,096 & 414 & 469 & 139 & 104 \\
    
    \bottomrule
    \end{tabular}
    }
    
    \label{Table:robust}
\end{table}
% \begin{table}
%     \centering
%     \scalebox{0.86}{
%     \begin{tabular}{cc|cccc}
%     \toprule
%     & &
%     \makecell{Diagonal\\$\eta=0$} &
%     \makecell{Diagonal\\$\eta=1$} &
%     \makecell{Triangular\\$\eta=0$} &
%     \makecell{Triangular\\$\eta=1$} \\
%     \cmidrule(lr){3-6}
%     Dimension & $\nu$ & \multicolumn{4}{c}{iteration-count range over four refinements} \\
%     \midrule
%     2D & $1$       & 159--169 & 183--192 & 64--65  & 45--47 \\
%     2D & $10^{-3}$ & 157--165 & 180--187 & 66--67  & 48--48 \\
%     2D & $10^{-6}$ & 157--165 & 180--187 & 49--51  & 36--37 \\
%     \midrule
%     3D & $1$       & 371--444 & 426--506 & 122--187 & 96--138 \\
%     3D & $10^{-3}$ & 362--414 & 414--469 & 143--203 & 114--151 \\
%     3D & $10^{-6}$ & 362--414 & 415--469 & 113--147 & 90--109 \\
%     \bottomrule
%     \end{tabular}
%     }
%     \caption{Iteration-count ranges for MINRES with~\eqref{preconditioner_diagonal_mass} and GMRES with~\eqref{preconditioner_triangular_generic}--\eqref{H_mass}. Each range is taken over the four mesh levels stated in the text.}
%     \label{Table:robust}
% \end{table}

The results show that, in both two and three dimensions, both preconditioners remain robust with respect to $\nu$ and $h$, regardless of whether the augmented Lagrangian term is included.
Overall, the block triangular preconditioner yields faster convergence than the block diagonal one, likely because it accounts for the off-diagonal coupling between the variables.
It is also interesting to note that the two preconditioners respond differently to the augmented Lagrangian term: the block diagonal preconditioner performs better for the non-augmented system, whereas the block triangular preconditioner performs better for the augmented system.
In fact, the augmented Lagrangian term is mainly a technical device in the analysis, introduced to exploit the discrete inf--sup condition~\eqref{LBB_algebraic}, as discussed in \cref{Rem:eta_zero}.
The numerical results indicate that the proposed preconditioners remain robust even when this term is omitted.
However, a theoretical explanation for this behavior is still lacking.

% Subsection: Effect of the augmented Lagrangian parameter
\subsection{Effect of the augmented Lagrangian parameter \texorpdfstring{$\eta$}{}}
To isolate the dependence on $\eta$, \cref{Table:eta} reports representative results on the two-dimensional mesh with 466{,}944 DOFs. We retain five values of $\eta$ spanning the regimes observed in the full parameter sweep.

% Table: eta
\begin{table}
    \centering
     \caption{Iteration counts in two dimensions with 466{,}944 DOFs for selected values of $\eta$ and $\nu$. The diagonal and triangular rows correspond to~\eqref{preconditioner_diagonal_mass} and~\eqref{preconditioner_triangular_generic}--\eqref{H_mass}, respectively.}
    \scalebox{0.86}{
    \begin{tabular}{c|c|ccccc}
    \toprule
    Preconditioner & \diagbox{$\eta$}{$\nu$} & $1$ & $10^{-1}$ & $10^{-2}$ & $10^{-3}$ & $10^{-4}$ \\
    \midrule
    \multirow{5}{*}{Diagonal}
    & $0$       & 168 & 164 & 164 & 164 & 164 \\
    & $10^{-1}$ & 171 & 167 & 167 & 166 & 166 \\
    & $1$       & 192 & 187 & 187 & 187 & 187 \\
    & $10^{1}$  & 360 & 345 & 344 & 344 & 344 \\
    & $10^{2}$  & 497 & 442 & 426 & 426 & 426 \\
    \midrule
    \multirow{5}{*}{Triangular}
    & $0$       & 65  & 67  & 67  & 67  & 66 \\
    & $10^{-1}$ & 62  & 64  & 64  & 64  & 64 \\
    & $1$       & 46  & 48  & 48  & 48  & 48 \\
    & $10^{1}$  & 84  & 82  & 80  & 80  & 80 \\
    & $10^{2}$  & 205 & 275 & 142 & 263 & 123 \\
    \bottomrule
    \end{tabular}
    }
   
    \label{Table:eta}
\end{table}

The block diagonal preconditioner performs better for smaller values of $\eta$, specifically for $\eta \le 10^{-1}$.
By contrast, the block triangular preconditioner attains its best performance when $\eta$ is around $1$.
For both preconditioners, excessively large values of $\eta$ result in poorer performance.
The same qualitative behavior was observed on the next refinement level and for the intermediate values of $\eta$ omitted from the table.

% Section: Conclusion
\section{Conclusion}
\label{Sec:Conclusion}
We proposed optimal block preconditioners for the MCS discretization of Stokes flow and provided rigorous analyses based on spectral analysis for the block diagonal preconditioners and FOV analysis for the block triangular preconditioners.
We also presented numerical results that demonstrate robustness with respect to the mesh size and kinematic viscosity.

As discussed in \cref{Sec:Numerical}, while we numerically observe that the proposed block diagonal preconditioner performs well even in the absence of the augmented Lagrangian term, i.e., when $\eta = 0$, the supporting theory remains open.

On the other hand, we note that there have been various algebraic approaches to designing block preconditioners for double saddle point problems in the literature~\cite{AB:2018,BBH:2024,HM:2019,XL:2020,WGS:2017}, and these ideas may be applied to our problem to construct alternative preconditioners.
The derivation of such alternatives and numerical comparisons among them are interesting directions for future work.

% Acknowledgement
\section*{Acknowledgements}
The authors would like to thank Prof. Joachim Sch\"{o}berl and Prof. Jinchao Xu for valuable discussions. The authors also acknowledge the use of the University of Oxford Advanced Research Computing~(ARC) facility in carrying out this work. \url{https://doi.org/10.5281/zenodo.22558}

\appendix

% Appendix: Proof of the discrete inf--sup condition
\section{Proof of the discrete inf--sup condition \texorpdfstring{\eqref{LBB_b1}}{}}
\label{App:LBB_b1}
In this appendix, we prove the discrete inf--sup condition~\eqref{LBB_b1} for the bilinear form $b_1 (\cdot , \cdot)$.
Namely, we prove
\begin{equation}
\label{LBB_b1_revisited}
\sup_{v_h\in V_h}
\frac{b_1(v_h,q_h)}{\|v_h\|_{V_h}}
\gtrsim
\|q_h\|_{L^2(\Omega)},
\quad q_h\in Q_h.
\end{equation}
Let $q_h\in Q_h$. The continuous right-inverse property of the divergence~(see, e.g.,~\cite[Lemma~11.2.3]{BS:2008}) provides $v\in H_0^1(\Omega)^d$ such that
\begin{equation}
\label{div_right_inverse}
\operatorname{div}v=q_h,
\quad
\|v\|_{H^1(\Omega)}\lesssim\|q_h\|_{L^2(\Omega)}.
\end{equation}
Let $\Pi_h$ be the canonical Raviart--Thomas interpolant and set 
\begin{equation*}
v_h=\Pi_hv\in V_h.
\end{equation*}
The Raviart--Thomas commuting property~\cite[Proposition~2.5.2]{BBF:2013} gives
\begin{equation}
\label{b1_qh}
b_1(v_h,q_h)
=(\operatorname{div}v_h,q_h)_{L^2(\Omega)}
=(\operatorname{div}v,q_h)_{L^2(\Omega)}
=\|q_h\|_{L^2(\Omega)}^2.
\end{equation}

Now, in view of~\eqref{LBB_b1_revisited} and~\eqref{b1_qh}, it remains to prove
\begin{equation}
\label{LBB_b1_intermediate}
\|v_h\|_{V_h} \lesssim \|q_h\|_{L^2(\Omega)}.
\end{equation}
Recall that the norm $\|\cdot\|_{V_h}$ is defined in~\eqref{norms}.
We first estimate the elementwise term in $\|v_h\|_{V_h}$.
For each element $T \in \mathcal T_h$, let $c_T = \frac{1}{|T|}\int_T v\,dx$.
In addition, let $\Pi_T$ denote the local Raviart--Thomas
interpolation on $T$.
Then it follows that
\begin{multline}
\label{element_estimate_local}
|v_h|_{H^1(T)}
= |v_h - c_T|_{H^1(T)}
\lesssim h^{-1} \|v_h - c_T\|_{L^2(T)} \\
= h^{-1} \|\Pi_T (v - c_T)\|_{L^2(T)}
\lesssim h^{-1} \|v - c_T\|_{L^2(T)} + |v|_{H^1(T)}
\lesssim |v|_{H^1(T)},
\end{multline}
where the first inequality follows from the inverse inequality~\cite[Lemma~4.9]{JPX:2025}, the second from the local $L^2$-stability of $\Pi_T$ (a direct consequence of \cite[Proposition~2.5.1]{BBF:2013}), and the last from the Poincar\'e inequality~\cite[Lemma~4.3]{JPX:2025}.
Summing~\eqref{element_estimate_local} over all $T \in \mathcal{T}_h$, we obtain
\begin{equation}
\label{element_estimate_global}
\sum_{T\in\mathcal T_h}\|\varepsilon(v_h)\|_{L^2(T)}^2
\le
\sum_{T\in\mathcal T_h}\|\nabla v_h\|_{L^2(T)}^2
\lesssim
|v|_{H^1(\Omega)}^2.
\end{equation}

Next, we estimate the facet jump term.
For each interior facet $F \in \mathcal{F}_h^{\mathrm{int}}$, let $T_1$ and $T_2$ be the two elements sharing $F$.
Then we have
\begin{multline}
\label{facet_estimate_local}
    h^{-1} \| \llbracket (v_h)_t \rrbracket \|_{L^2(F)}^2
    = h^{-1} \| \llbracket (v_h - v)_t \rrbracket \|_{L^2(F)}^2
    \lesssim \sum_{j=1}^2 h^{-1} \| ((v_h)_{T_j} - v)_t \|_{L^2(F)}^2 \\
    \lesssim \sum_{j=1}^2 \left( h^{-2} \| v_h - v \|_{L^2(T_j)}^2 + | v_h - v |_{H^1(T_j)}^2 \right)
    \lesssim \sum_{j=1}^2 | v |_{H^1(T_j)}^2,
\end{multline}
where the second inequality follows from the trace inequality~\cite[Lemma~4.5]{JPX:2025}, and the last inequality follows from the local interpolation error estimate~\cite[Proposition~2.5.1]{BBF:2013}.
The boundary facet jump term for $F \in \mathcal{F}_h^{\mathrm{bd}}$ can be estimated similarly to~\eqref{facet_estimate_local}.
Summing over all facets, we obtain
\begin{equation}
\label{facet_estimate_global}
\sum_{F \in \mathcal{F}_h}
h^{-1} \| \llbracket (v_h)_t \rrbracket \|_{L^2(F)}^2
\lesssim
|v|_{H^1(\Omega)}^2.
\end{equation}

Combining~\eqref{element_estimate_global} and~\eqref{facet_estimate_global}, we obtain
\begin{equation*}
\|v_h\|_{V_h}^2
\stackrel{\eqref{norms}}{=}
\sum_{T \in \mathcal{T}_h} \|\varepsilon(v_h)\|_{L^2(T)}^2
+
\sum_{F \in \mathcal{F}_h} h^{-1} \|\llbracket (v_h)_t \rrbracket\|_{L^2(F)}^2
\lesssim |v|_{H^1(\Omega)}^2
\stackrel{\eqref{div_right_inverse}}{\lesssim} \| q_h \|_{L^2(\Omega)}^2,
\end{equation*}
which proves~\eqref{LBB_b1_intermediate}.
This completes the proof.

% Appendix: Proof of LVh
\section{Proof of \texorpdfstring{\cref{Lem:LVh}}{LVh}}
\label{App:LVh}
In this appendix, we present a proof of \cref{Lem:LVh}, i.e., the optimal preconditioning property of the auxiliary space preconditioner (see~\cite{PX:2025,Xu:1996}) $L_{V_h}$ defined in~\eqref{LVh_aux} for the matrix $M_{V_h}$.

We first present a lemma concerning a quasi-interpolation operator from $V_h$ to $\tilde{V}_h$, where $\tilde{V}_h$ is defined in~\eqref{Vh_tilde}.

% Lemma: Quasi-interpolation
\begin{lemma}
\label{Lem:quasi-interpolation}
There exists a linear operator $\tilde{I}_h \colon V_h \to \tilde{V}_h$ such that
\begin{equation*}
(\tilde{I}_h v,\tilde{I}_h v)_{\tilde{V}_h}
+ \sum_{T \in \mathcal{T}_h} h^{-2}\|v-\tilde{I}_h v\|_{L^2(T)}^2
\lesssim
\|v\|_{V_h}^2,
\quad
v \in V_h,
\end{equation*}
where $V_h$, $\| \cdot \|_{V_h}$, $\tilde{V}_h$, and $(\cdot,\cdot)_{\tilde{V}_h}$ are defined in~\eqref{spaces},~\eqref{norms},~\eqref{Vh_tilde}, and~\eqref{M_Vh_tilde}, respectively.
\end{lemma}
\begin{proof}
We set 
\begin{equation*}
    \tilde{I}_h = J_{h0}^{\mathrm{av}} \circ I_h^\#,
\end{equation*}
where $J_{h0}^{\mathrm{av}}$ is the averaging operator onto $\tilde{V}_h$ defined in~\cite[equation~(6.7)]{EG:2017} and $I_h^\#$ is the local interpolation operator onto $\mathbb{P}^1 (\mathcal{T}_h, \mathbb{R}^d)$ defined in~\cite[Section~5]{EG:2017}.

We first summarize some known approximation properties of $I_h^{\#}$ and $J_{h0}^{\mathrm{av}}$.
In~\cite[Proposition~3.1 and Theorem~3.3]{EG:2017}, the following local approximation property of $I_h^{\#}$ is shown:
\begin{equation}
\label{Ih_sharp}
    h^{-2} \| v - I_h^{\#} v \|_{L^2(T)}^2 + | v - I_h^{\#} v |_{H^1(T)}^2 + | I_h^{\#} v |_{H^1(T)}^2 \lesssim | v |_{H^1(T)}^2,
    \quad v \in V_h
\end{equation}
Moreover, $J_{h0}^{\mathrm{av}}$ satisfies the following estimate~\cite[Lemma~6.2]{EG:2017}:
\begin{equation}
\label{Jh0_av}
    \sum_{T \in \mathcal{T}_h} h^{-2} \| w - J_{h0}^{\mathrm{av}} w \|_{L^2(T)}^2
    \lesssim \sum_{F \in \mathcal{F}_h} h^{-1} \| \llbracket w \rrbracket \|_{L^2(F)}^2,
    \quad w \in \mathbb{P}^1(\mathcal{T}_h, \mathbb{R}^d).
\end{equation}

Take any $v \in V_h$.
It follows that
\begin{equation}
\label{Lem2:quasi-interpolation}
\begin{aligned}
    (\tilde{I}_h v, \tilde{I}_h v)_{\tilde{V}_h}
    &\le |\tilde{I}_h v|_{H^1(\Omega)}^2 \\
    &\lesssim \sum_{T \in \mathcal{T}_h} \left( | I_h^{\#} v |_{H^1(T)}^2 + | I_h^{\#} v - J_{h0}^{\mathrm{av}} I_h^{\#} v |_{H^1(T)}^2 \right) \\
    &\lesssim \sum_{T \in \mathcal{T}_h} \left( | I_h^{\#} v |_{H^1(T)}^2 + h^{-2} \| I_h^{\#} v - J_{h0}^{\mathrm{av}} I_h^{\#} v \|_{L^2(T)}^2 \right),
\end{aligned}
\end{equation}
where the last inequality follows from the inverse inequality~\cite[Lemma~4.9]{JPX:2025}.
Moreover, for each $T \in \mathcal{T}_h$, we have
\begin{equation}
\label{Lem3:quasi-interpolation}
    h^{-2} \| v - \tilde{I}_h v \|_{L^2(T)}^2
    \lesssim h^{-2} \| v - I_h^{\#} v \|_{L^2(T)}^2 + h^{-2} \| I_h^{\#} v - J_{h0}^{\mathrm{av}} I_h^{\#} v \|_{L^2(T)}^2.
\end{equation}
Combining~\eqref{Lem2:quasi-interpolation} and~\eqref{Lem3:quasi-interpolation}, we obtain
\begin{equation}
\label{Lem4:quasi-interpolation}
\begin{aligned}
    (\tilde{I}_h v,\tilde{I}_h v)_{\tilde{V}_h}
&+ \sum_{T \in \mathcal{T}_h} h^{-2}\|v-\tilde{I}_h v\|_{L^2(T)}^2 \\
&\lesssim \sum_{T \in \mathcal{T}_h} \left( h^{-2} \| v - I_h^{\#} v \|_{L^2(T)}^2 + h^{-2} \| I_h^{\#} v - J_{h0}^{\mathrm{av}} I_h^{\#} v \|_{L^2(T)}^2 + | I_h^{\#} v |_{H^1(T)}^2 \right) \\
&\stackrel{\eqref{Jh0_av}}{\lesssim} \sum_{T \in \mathcal{T}_h} \left( h^{-2} \| v - I_h^{\#} v \|_{L^2(T)}^2 + | I_h^{\#} v |_{H^1(T)}^2 \right)
+ \sum_{F \in \mathcal{F}_h} h^{-1} \| \llbracket I_h^{\#} v \rrbracket \|_{L^2(F)}^2.
\end{aligned}
\end{equation}

We estimate the interface term on the right-hand side of~\eqref{Lem4:quasi-interpolation}.
For $F \in \mathcal{F}_h^{\mathrm{int}}$,
\begin{multline}
\label{Lem6:quasi-interpolation}
\sum_{F \in \mathcal{F}_h^{\mathrm{int}}} h^{-1} \| \llbracket I_h^{\#} v \rrbracket \|_{L^2(F)}^2
\lesssim \sum_{F \in \mathcal{F}_h^{\mathrm{int}}} \left( h^{-1} \| \llbracket v \rrbracket \|_{L^2(F)}^2 + h^{-1} \| \llbracket v - I_h^{\#} v \rrbracket \|_{L^2(F)}^2 \right) \\
\lesssim \sum_{F \in \mathcal{F}_h^{\mathrm{int}}} h^{-1} \| \llbracket v \rrbracket \|_{L^2(F)}^2
+ \sum_{T \in \mathcal{T}_h} \left( h^{-2} \| v - I_h^{\#} v \|_{L^2(T)}^2 + | v - I_h^{\#} v |_{H^1(T)}^2 \right),
\end{multline}
where the last inequality follows from the trace inequality~\cite[Lemma~4.5]{JPX:2025}.
The corresponding $\sum_{F \in \mathcal{F}_h^{\mathrm{bd}}}$ term can be estimated similarly, and we omit the details.

Combining~\eqref{Lem4:quasi-interpolation} and~\eqref{Lem6:quasi-interpolation}, we obtain
\begin{equation*}
\begin{aligned}
    &(\tilde{I}_h v,\tilde{I}_h v)_{\tilde{V}_h}
+ \sum_{T \in \mathcal{T}_h} h^{-2}\|v-\tilde{I}_h v\|_{L^2(T)}^2 \\
&\lesssim \sum_{T \in \mathcal{T}_h} \left( h^{-2} \| v - I_h^{\#} v \|_{L^2(T)}^2 + | v - I_h^{\#} v |_{H^1(T)}^2 + | I_h^{\#} v |_{H^1(T)}^2 \right) + \sum_{F \in \mathcal{F}_h} h^{-1} \| \llbracket v \rrbracket \|_{L^2(F)}^2 \\
&\stackrel{\eqref{Ih_sharp}}{\lesssim} \sum_{T \in \mathcal{T}_h} | v |_{H^1(T)}^2 + \sum_{F \in \mathcal{F}_h} h^{-1} \| \llbracket v \rrbracket \|_{L^2(F)}^2.
\end{aligned}
\end{equation*}
Finally, applying the discrete Korn inequality~\cite[equation~(4.8)]{GJPS:2020b}, and observing that $\llbracket v \rrbracket = \llbracket v_t \rrbracket$
because $v \in H_0(\operatorname{div};\Omega)$, completes the proof.
\end{proof}

Next, we observe that the matrix $D_{V_h}$, the diagonal of $M_{V_h}$, satisfies
\begin{equation}
\label{DVh}
\| v \|_{V_h}^2 \lesssim (D_{V_h} v, v)
\lesssim h^{-2} \| v \|_{L^2(\Omega)}^2,
\end{equation}
where the first inequality follows from a standard coloring argument, and the second inequality follows from the inverse inequality, the trace inequality, and the finite overlap property.

The preconditioner $L_{V_h}$ in~\eqref{LVh_aux} has the auxiliary space form (see~\cite{PX:2025})
\begin{equation}
\label{aux_structure}
L_{V_h} = \Pi \undertilde{L} \Pi^T,
\quad
\Pi = \begin{bmatrix} I & I_h \end{bmatrix},
\quad
\undertilde{L} =
\begin{bmatrix}
D_{V_h}^{-1} & 0 \\
0 & L_{\tilde{V}_h}
\end{bmatrix}.
\end{equation}
Thanks to~\eqref{aux_structure}, we can utilize the following sharp estimate for the extremal eigenvalues of auxiliary space preconditioners, presented in~\cite[Theorem~3.6]{PX:2025}.

% Lemma: Condition number
\begin{lemma}
\label{Lem:condition_number_aux}
Let $V$ and $\undertilde{V}$ be finite-dimensional vector spaces, and let $\Pi \colon \undertilde{V} \to V$ be a surjective linear operator.
Let $M \colon V \to V$ and $\undertilde{L} \colon \undertilde{V} \to \undertilde{V}$ be SPD linear operators, and let
\begin{equation*}
    L = \Pi \undertilde{L} \Pi^T \colon V \to V.
\end{equation*}
Then $L$ is SPD and satisfies
\begin{align*}
\lambda_{\min}(LM) &= \left( \sup_{v \in V,\ \| v \|_{M} = 1} \inf_{\undertilde{v} \in \undertilde{V},\ \Pi \undertilde{v} = v} ( \undertilde{L}^{-1} \undertilde{v}, \undertilde{v} ) \right)^{-1}, \\
\lambda_{\max}(LM) &= \left( \inf_{v \in V,\ \| v \|_{M} = 1} \inf_{\undertilde{v} \in \undertilde{V},\ \Pi \undertilde{v} = v} ( \undertilde{L}^{-1} \undertilde{v}, \undertilde{v} ) \right)^{-1}.
\end{align*}
\end{lemma}

Now, we are ready to prove \cref{Lem:LVh}.

\begin{proof}[Proof of \cref{Lem:LVh}]
Note that \eqref{aux_structure} implies
\[
(\undertilde{L}^{-1} \undertilde{v}, \undertilde{v} )
= (D_{V_h} v, v) + ( L_{\tilde{V}_h}^{-1} \tilde{v}, \tilde{v}),
\quad
\undertilde{v} = (v,\tilde{v})\in V_h \times \tilde{V}_h.
\]
By \cref{Lem:condition_number_aux}, it is enough to verify the following:
\begin{itemize}
\item[(a)] For any $v\in V_h$ and $\tilde{v}\in \tilde{V}_h$, we have
\begin{equation}
\label{Lem1:LVh}
\|v+I_h \tilde{v}\|_{V_h}^2
\lesssim (D_{V_h} v, v) + ( L_{\tilde{V}_h}^{-1} \tilde{v}, \tilde{v})
\end{equation}
\item[(b)] For any $w \in V_h$, there exist $v\in V_h$ and $\tilde{v}\in \tilde{V}_h$ such that
$w = v+I_h \tilde{v}$ and
\begin{equation}
\label{Lem2:LVh}
(D_{V_h} v, v) + ( L_{\tilde{V}_h}^{-1} \tilde{v}, \tilde{v})
\lesssim \|w\|_{V_h}^2.
\end{equation}
\end{itemize}

We first prove~\eqref{Lem1:LVh} as follows:
\[
\|v+I_h \tilde{v}\|_{V_h}^2
\lesssim
\|v\|_{V_h}^2 + \|\tilde{v}\|_{\tilde{V}_h}^2
\lesssim
( D_{V_h} v, v) + (L_{\tilde{V}_h}^{-1} \tilde{v}, \tilde{v}),
\]
where the first inequality follows from the fact that $I_h$ is the natural embedding,
and the second inequality follows from~\eqref{LVh_tilde} and~\eqref{DVh}.

Next, we prove \eqref{Lem2:LVh}.
Given $w\in V_h$, we choose
\[
\tilde{v} = \tilde{I}_h w,
\quad
v = w - I_h\tilde{I}_h w.
\]
Then clearly $w = v + I_h \tilde{v}$.
Moreover, by \cref{Lem:quasi-interpolation}, we get
\[
( D_{V_h} v, v)
= ( D_{V_h} (w -I_h \tilde{I}_h w), w -I_h \tilde{I}_h w)
\stackrel{\eqref{DVh}}{\lesssim}
h^{-2}\| w - \tilde{I}_h w\|_{L^2(\Omega)}^2
\lesssim
\|w\|_{V_h}^2.
\]
On the other hand, again by \cref{Lem:quasi-interpolation},
\[
(L_{\tilde{V}_h}^{-1} \tilde{v}, \tilde{v})
=
(L_{\tilde{V}_h}^{-1} \tilde{I}_h w, \tilde{I}_h w)
\stackrel{\eqref{LVh_tilde}}{\lesssim}
\|\tilde{I}_h w \|_{\tilde{V}_h}^2
\lesssim
\|w \|_{V_h}^2.
\]
Thus \eqref{Lem2:LVh} holds.
The proof is complete.
\end{proof}

% % Remark: High-order auxiliary space
% \begin{remark}
% \label{Rem:LVh_high_order}
% Alternatively to~\eqref{Vh_tilde}, we may choose a high-order auxiliary conforming space
% \[
% \tilde{V}_h = H_0^1(\Omega, \mathbb{R}^d)\cap \mathbb{P}^r(\mathcal{T}_h,\mathbb{R}^d),
% \quad 1 \le r \le k,
% \]
% and this choice can be analyzed similarly.
% However, the lowest-order choice~\eqref{Vh_tilde} is already sufficient to obtain an $h$-uniform
% auxiliary-space preconditioner for $M_{V_h}$. Since the purpose of the auxiliary
% space is to admit a cheap optimal solver rather than to match the polynomial degree
% of the original Raviart--Thomas space, the choice $r=1$ is typically the most
% economical in practice.
% \end{remark}

\bibliographystyle{siamplain}
\bibliography{refs_Stokes}

\end{document}